\documentclass[11pt,reqno,twosdie]{amsart}

\usepackage{amsmath,amssymb,mathrsfs}

\usepackage[asymmetric,top=3.5cm,bottom=4.3cm,left=3.1cm,right=3.1cm]{geometry}
\usepackage{amsmath,amsfonts,amsthm,mathrsfs,amssymb,cite}
\usepackage[usenames]{color}

\usepackage{mathtools} 
\usepackage{graphics}
\usepackage{framed}

\usepackage{enumerate}

\usepackage{hyperref}
\usepackage{xcolor}
\hypersetup{
  colorlinks,
  linkcolor={blue!80!black},
  urlcolor={blue!80!black},
  citecolor={blue!80!black}
}
\usepackage[capitalize,noabbrev]{cleveref}

\usepackage{graphicx}
\usepackage{latexsym} 
\usepackage{enumerate}

\theoremstyle{plain}
\newtheorem{theorem}{Theorem}[section]
\newtheorem{lemma}[theorem]{Lemma}
\newtheorem{corollary}[theorem]{Corollary}
\newtheorem{proposition}[theorem]{Proposition}

\theoremstyle{remark}
\newtheorem{definition}[theorem]{Definition}
\newtheorem{remark}[theorem]{Remark}

\renewcommand{\eqref}[1]{\textnormal{(\ref{#1})}}

\numberwithin{equation}{section}

\newcommand{\bfx}{\mathbf{x}}

\newcommand{\R}{\mathbb{R}}

\newcommand{\N}{\mathbb{N}}

\def\bsi{{\mathrm{i}}}

\def\bsl{ \boldsymbol {l} }
 \def\bsL{ \boldsymbol {L} }
\def\bfx{ \mathbf {x} }
 
\usepackage{color}

\title[On nodal and generalized singular structures of Laplacian eigenfunctions]{On nodal and generalized singular structures of Laplacian eigenfunctions and applications in $\mathbb{R}^3$} 

\author{Xinlin Cao}
\address{Department of Mathematics, Hong Kong Baptist University, Kowloon, Hong Kong, China.}
\email{xlcao.math@foxmail.com}

\author{Huaian Diao}
\address{School of Mathematics and Statistics, Northeast Normal University,
Changchun, Jilin 130024, China.}
\email{hadiao@nenu.edu.cn}

\author{Hongyu Liu}
\address{Department of Mathematics, Hong Kong Baptist University, Kowloon, Hong Kong, China.}
\email{hongyu.liuip@gmail.com, hongyuliu@hkbu.edu.hk}

\author{Jun Zou}
\address{Department of Mathematics, The Chinese University of Hong Kong, Hong Kong, China}
\email{zou@math.cuhk.edu.hk}

\begin{document}

\begin{abstract}

This paper is a continuation and an extension of our recent work \cite{CDL2} on the geometric structures of Laplacian eigenfunctions and their applications to inverse scattering problems. In \cite{CDL2}, the analytic behaviour of the Laplacian eigenfunctions is investigated at a point where two nodal or generalised singular lines intersect. It reveals a certain intriguing property that the vanishing order of the eigenfunction at the intersecting point is related to the rationality of the intersecting angle. In this paper, we consider the 3D counterpart of such a study by studying the analytic behaviours of the Laplacian eigenfunctions at places where nodal or generalised singular planes are intersect. Compared to the 2D case, the geometric situation is more complicated: the intersection of two planes generates an edge corner, whereas the intersection of more than three planes generates a vertex corner. We provide a comprehensive characterisation for all of those cases. Moreover, we apply the spectral results to establish some novel unique identifiability results for the geometric inverse problems of recovering the shape as well as the (possible) surface impedance parameter by the associated scattering far-field measurements.

\medskip 
\noindent{\bf Keywords} Laplacian eigenfunctions, geometric structures, nodal and generalised singular planes, inverse scattering, impedance obstacle, uniqueness, a single far-field pattern
 
\medskip
\noindent{\bf Mathematics Subject Classification (2010)}: 35P05, 35P25, 35R30, 35Q60

\end{abstract}

\maketitle

\section{Introduction}\label{sec:Intro}

In this paper, we are concerned with the geometric structures of Laplacian eigenfunctions and their application to the geometrical inverse scattering problem. The study of the geometric properties of Laplacian eigenfunctions has a rich theory in the literature, and it still remains an active field with many colourful theoretical and computational developments. We refer to our recent paper \cite{CDL2} and the related references therein for a relatively comprehensive introduction on this intriguing topic. In fact, the current article is a continuation and an extension of our study in \cite{CDL2}, where the intersection of two nodal or generalised singular lines is considered. It reveals certain intriguing property that the vanishing order (analytic quantity) of the eigenfunction at the intersecting point is related to the rationality (geometric quantity) of the intersecting angle. This spectral result is applied directly to the inverse obstacle scattering problem and the inverse diffraction grating problem in establishing several novel unique identifiability results in determining the polygonal shape/support of an inhomogeneous scattering object as well as the (possible) surface impedance parameter by at most a few far-field measurements. It is natural to consider the corresponding extension to the three-dimensional setting by studying the intersections of nodal or generalised singular planes and their implications to the analytic behaviours of the eigenfunctions. In three dimensions, the geometric setup is more complicated: the intersection of two planes produces an edge corner, whereas the intersection of more than three planes produces a vertex corner; see Fig.~1 in what follows for a schematic illustration. We shall derive comprehensive characterisation on the relationship between the analytic behaviours of the eigenfunction at the corner point and the geometric quantities of that corner. Indeed, in the former case, we show that the vanishing order of the eigenfunction is related to the rationality of the intersecting angle in a similar manner to the two-dimensional case, whereas in the latter case, the vanishing order of the eigenfunction is related to the intersecting angle in a more complicated and mysterious manner through the roots of the Legendre polynomials. Similar to \cite{CDL2}, the obtained spectral results are also applied to derive several novel unique identifiability results in the geometrical inverse scattering problem of determining an impenetrable obstacle as well as the (possibly) surface impedance by at most a few far-field measurements in the polyhedral setup. The rest of this section is mainly devoted to the introduction of the mathematical setup for our study. 

Let  $\Omega$ be an open set in $\mathbb{R}^3$. Consider $u\in L^2(\Omega)$ and $\lambda\in\mathbb{R}_+$ such that
\begin{equation}\label{eq:eig}
-\Delta u=\lambda u. 
\end{equation}
$u$ in \eqref{eq:eig} is referred to as a (generalised) Laplacian eigenfunction. Indeed, compared to the conventional notion of Laplacian eigenfunctions, we do not prescribe any homogeneous boundary condition on $\partial\Omega$ for $u$ in \eqref{eq:eig}. That means, the spectral results that we establish in this paper apply to any function that satisfies \eqref{eq:eig} in the interior of $\Omega$, in particular, including all the conventional Laplacian eigenfunctions. We next introduce several critical definitions for our subsequent use. In what follows, for $\Pi$ being a flat plane in $\mathbb{R}^3$, any non-empty open connected subset $\Sigma\Subset\Pi$ is called a \emph{cell} of $\Pi$. Let $\widetilde{\Pi}=\Pi_\Sigma$ denote the connected component of $\Pi\cap\Omega$ that contains $\Sigma$. 

\begin{definition}\label{def:1}
Consider $u$ to \eqref{eq:eig} being a nontrivial eigenfunction. Let $\Sigma\subset\Omega$ be a cell of $\Pi$, and let $\eta\in\mathbb{C}$ be a constant. If $u|_{\Sigma}=0$, $\Sigma$ is said to be a nodal cell of $u$ in $\Omega$. By analytic continuation, it is seen that $u|_{\widetilde{\Pi}}=0$, and $\widetilde{\Pi}$ is said to be a nodal plane of $u$. In a similar manner, in the case $(\partial_\nu u+\eta u)\big|_{\Sigma}=0$, where $\nu$ is a unit one-sided normal direction of $\Pi$, $\Sigma$ and $\widetilde\Pi$ are respectively called the generalised singular cell and plane. In the particular case $\eta\equiv 0$, a generalised singular plane is also called a singular plane. Let ${\mathcal N}^\lambda_\Omega$, ${\mathcal S}^\lambda_\Omega $ and ${\mathcal M}^\lambda_\Omega$, respectively, signify the sets of nodal, singular and generalised singular planes of $u$ in \eqref{eq:eig}. 
\end{definition}

According to Definition~\ref{def:1}, a nodal/generalised singular plane is actually a cell that is fully extended in $\Omega$. Indeed, by the fact that $u$ is analytic in $\Omega$, we know that if the homogeneous condition is satisfied on a cell, then it is also satisfied on the so-called ``plane" in Definition~\ref{def:1} by the analytic continuation. In what follows, most of the planes are actually the nodal/generalised singular planes in the sense of Definition~\ref{def:1}, which should be clear from the context. Let $B_\rho(\mathbf{x})$ denote a ball of radius $\rho\in\mathbb{R}_+$ and centred at $\mathbf{x}\in\mathbb{R}^3$. For a set $K\subset\mathbb{R}^3$, $B_\rho(K):=\{\mathbf{x}; \mathbf{x}\in B_\rho(\mathbf{y})\ \mbox{for any}\ \mathbf{y}\in K\}$. 

\begin{definition}\label{def:edge}
Let $\Pi_j$, $j=1,2$, be two planes in $\Omega$ such that $\Pi_1\cap\Pi_2=\bsL$ with $\bsL$ a line segment. Let $\bsl\Subset\bsL$ be an open line segment and $\rho\in\mathbb{R}_+$ be sufficiently small such that $B_\rho(\bsl)\subset\Omega$. Let $\mathcal{W}(\Pi_1, \Pi_2)$ denote one of the wedge domains formed by $\Pi_1$ and $\Pi_2$. Then 
$\mathcal{W}(\Pi_1,\Pi_2)\cap B_\rho(\bsl)$ is called an edge corner associated with $\Pi_1$ and $\Pi_2$; see Fig.~1 for a schematic illustration. It is denoted by ${\mathcal E}(\Pi_1, \Pi_2,\bsl)$. Any $\mathbf{x}\in \bsl$ is said to be an edge-corner point of ${\mathcal E}(\Pi_1, \Pi_2,\bsl)$. 
\end{definition}

\begin{definition}\label{def:vertex}
Let $\{\Pi_\ell \}_{\ell=1}^n$  ($n\geq 3$)  be $n$ planes in $\Omega$ such that they form a  polyhedral cone $\mathcal K$ with the vertex $\bfx_0\in\Omega$. Let $\rho\in\mathbb{R}_+$ be sufficiently small such that $B_\rho(\bfx_0)\subset\Omega$. Then $\mathcal K \cap B_\rho(\bfx_0)$ is called a vertex corner associated with $\Pi_1$, $\Pi_2$, $\ldots$, $\Pi_n$. It is  denoted by $\mathcal{V}(\{\Pi_\ell\}_{\ell=1}^n, \bfx_0)$; see Fig.~1 for a schematic illustration. 

\end{definition}

\begin{figure}[htbp]
 \includegraphics[width=0.4\linewidth]{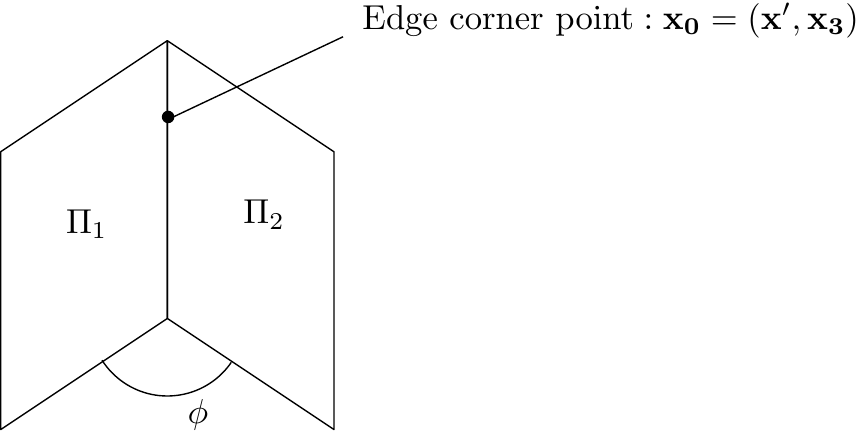} \qquad \includegraphics[width=0.3\linewidth]{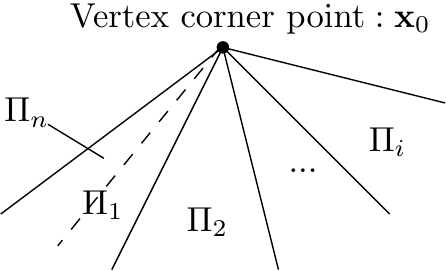}
 \caption{Schematic illustrations of 3D edge corner and vertex corner respectively.}
\end{figure}

	It is obvious that a vertex corner $\mathcal{V}(\{\Pi_\ell\}_{\ell=1}^n, \bfx_0)$ is composed by finite many edge corners, which are intersected by any two adjacent pieces of planes. Moreover, a vertex-corner point must be an edge-corner point. Definitions~\ref{def:1}--\ref{def:vertex} fix some geometric notions. Next, we introduce several analytic notions for the Laplacian eigenfunction. 

\begin{definition}\label{def:3}
Let $u$ be a nontrivial eigenfunction in \eqref{eq:eig}. For a given point $\mathbf{x}_0 \in \Omega$, if there  exits a number $N \in {\mathbb N}\cup\{0\}$ such that  
\begin{equation}\label{eq:normal3}
	\lim_{\rho\rightarrow +0} \frac{1}{\rho^m} \int_{B_\rho(\mathbf{x}_0)}\, |u(\mathbf{x})|\, {\rm d} \mathbf{x}=0\ \ \mbox{for}\ \ m=0,1,\ldots, {{N+1}},
\end{equation}
we say that $u$ vanishes at $\mathbf{x}_0$ up to the order $N$. The largest possible $N$ such that \eqref{eq:normal3} is fulfilled is called the vanishing order of $u$ at $\mathbf{x}_0$, and we write 
\begin{equation*}\label{eq:normal4}
\mathrm{Vani}(u; \mathbf{x}_0)=N. 
\end{equation*}
If \eqref{eq:normal3} holds for any $N\in\mathbb{N}$, then we say that the vanishing order is infinity. 
\end{definition}

By the strong UCP, if the vanishing order of $u$ at $\bfx_0\in \Omega$ is infinite, we know that $u \equiv 0$  in $\Omega $.


For the definition of the vanishing order of $u$ at a 3D edge or vertex corner point, we have
\begin{definition}\label{def:4}
Let $u$ be a nontrivial eigenfunction in \eqref{eq:eig}. Consider an edge corner 
${\mathcal E}(\Pi_1, \Pi_2,\bsl)\Subset \Omega$.  For {{any}} given $\mathbf{x}_0 \in  \bsl$, if 
$$
\mathrm{Vani}(u; \mathbf{x}_0)=N, 
$$
we say that $u$ vanishes at $\bfx_0$ associated with the edge corner 
${\mathcal E}(\Pi_1, \Pi_2,\bsl)\Subset \Omega$ up to order $N$, denoted by
%
%
\begin{equation*}\label{eq:normal4}
\mathrm{Vani}(u; \bfx_0, \Pi_1, \Pi_2)=N. 
\end{equation*}
For a vertex-corner point $\mathbf{ x}_0\in \Omega$ which is intersected by $\Pi_i$, $i=1,2,...n$, the vanishing order of $u$ at $\mathbf{ x}_0$ is defined by
\begin{equation*}\label{def:3d plane}
\mathrm{Vani}(u; \mathbf{ x}_0) :=\max\big\{\max_{i=1,2,...n-1}\mathrm{Vani}(u; \bfx_0,\Pi_i, \Pi_{i+1}), \mathrm{Vani}(u; \bfx_0, \Pi_n, \Pi_1)\big\}. 
\end{equation*}
\end{definition}

With the above definitions, we are mainly concerned with the vanishing properties of the Laplacian eigenfunctions at places where two or more nodal/singular/generalised singular planes intersect. 
The remaining part of the paper is organised as follows. In Section \ref{sec2}, we consider the vanishing property of the Laplacian eigenfunction at an edge corner intersected by two planes of the three kinds: nodal plane, singular plane or generalized singular plane. In Section \ref{sec3}, we study the vanishing property at a vertex corner intersected by $n$ planes, $n\geq3$, on the basis of Section \ref{sec2}. As a direct consequence of Sections \ref{sec2} and \ref{sec3}, Section \ref{sec4} is devoted to the discussion of the irrational intersection as a special case with infinite vanishing order. In Section \ref{sec5}, as an important application of the obtained spectral results, we establish the unique identifiability results in determining the obstacle as well as the surface impedance parameter by at most two far-field measurements for the inverse obstacle scattering problem.


\section{Vanishing orders at edge-corner points}\label{sec2}


In this section, we study the vanishing property of the Laplacian eigenfunction at an edge-corner point $\bfx_0\in \bsl$ associated with ${\mathcal E}(\Pi_1, \Pi_2,\bsl)$. The two planes $\Pi_\ell\,(\ell=1,2)$ could be either one of following three types: nodal, singular or generalized singular.  First, we give a definition of the irrational or rational dihedral angle of two intersecting planes. 

\begin{definition}\label{def:2}
Let $P_1$ and $P_2$ be two planes in $\R^3$ that intersect with each other. Let  $\phi\in(0, \pi)$ be one of  the associated intersecting dihedral angle of $P_1$ and $P_2$. Set
\begin{equation*}\label{eq:normal2}
\phi=\alpha\cdot \pi, \ \ \alpha\in(0,1). 
\end{equation*}
Then, $\phi$ is said to be an \emph{irrational dihedral angle} if $\alpha$ is an irrational number; and it is said to be a rational dihedral angle of degree $q$ if $\alpha=p/q$ with $p, q\in\mathbb{N}$ and irreducible. 
\end{definition}


Since $-\Delta$ is invariant under rigid motions, throughout the rest of this paper, we assume that the edge corner ${\mathcal E}(\Pi_1, \Pi_2,\bsl)$ satisfies

\begin{equation*}\label{notation}
\bsl=\big\{~\bfx=(\bfx',x_3)\in \R^3; \bfx'=0,\ x_3\in (-H, H)\big\}\Subset\Omega,
\end{equation*}
where $2H$ is the length of $\bsl$. That is, $\bsl$ coincides with the $x_3$-axis. We further assume that $\Pi_1$ coincides with the $(x_1,x_2)$-plane while $\Pi_2$ possesses a dihedral angle $\alpha\cdot\pi$ away from $\Pi_1$ in the anti-clockwise direction; see Figure 2 for a schematic illustration. Clearly, we can assume that $\alpha\in(0, 1)$. 
 Moreover, when we consider the vanishing order at an edge-corner point of ${\mathcal E}(\Pi_1, \Pi_2,\bsl)$, throughout this section, we assume that the edge-corner point under consideration is the origin $\mathbf 0 \in \bsl$. 



In the next subsection, we first consider a relatively simpler case that at least one of the intersecting planes of ${\mathcal E}(\Pi_1, \Pi_2,\bsl)$ is a nodal plane. Without loss of generality, throughout this subsection, we assume that $u|_{\Pi_1}\equiv0$.   

\subsection{Vanishing orders at an edge-corner point with at least one plane being nodal}\label{vanish nodal2}

\begin{figure}[htbp]
	\centering
	\includegraphics[width=0.3\linewidth]{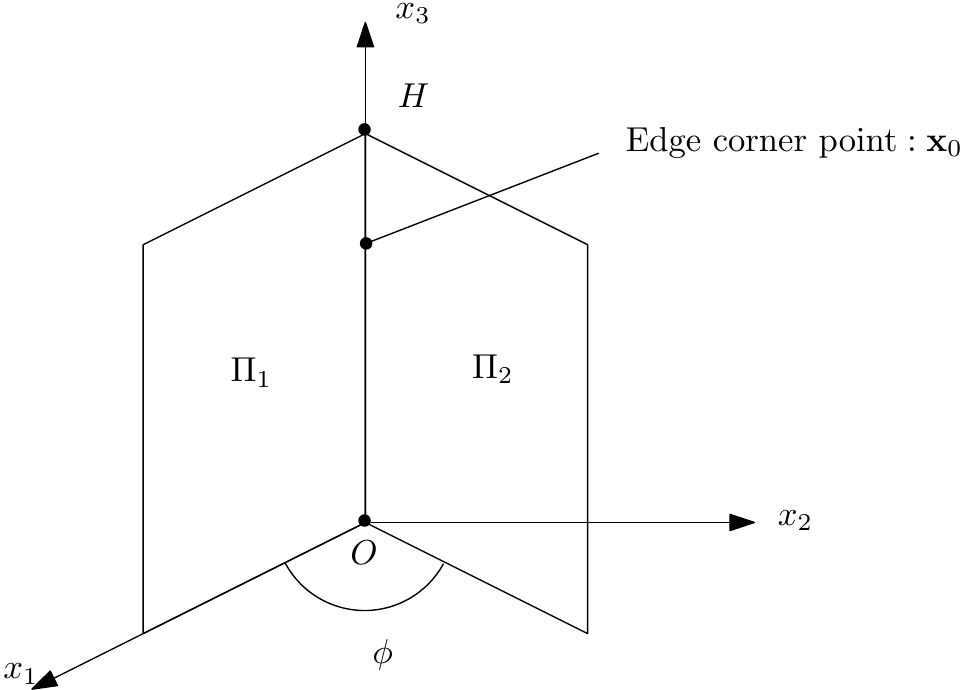}
	\caption{Schematic illustration of two intersecting planes with the 3D edge corner and the dihedral angle $\phi$.}
	\label{fig:coordinate1}
\end{figure}

We first derive several important auxiliary results for the subsequent use. 

\begin{proposition}\label{neumann bd}
	Let 
	\begin{equation}\label{eq:x sph}
		 \mathbf{ x}=(r\sin\theta\cos\phi, r\sin\theta\sin\phi, r\cos\theta):=(r,\theta,\phi),\ r\geq 0, \, {{\theta\in [0,\pi),\, \phi \in [0, 2\pi)}}
	\end{equation}
	 be the spherical  coordinate of $\bfx \in \R^3$. Let $\Pi$ be any of the two planes associated with ${\mathcal E}(\Pi_1, \Pi_2,\bsl)$. For any point $\bfx$ belonging to $\Pi$, we know that  $\phi$ defined in \eqref{eq:x sph} is fixed; see Fig.~2. Let $\nu$ be the unit normal vector that is perpendicular to $\Pi$. Then 
	 \begin{equation*}\label{neumann bdd}
	 \frac{\partial u}{\partial\nu}=\pm\frac{1}{r\sin\theta}\frac{\partial u}{\partial\phi}.	 
	 \end{equation*}
\end{proposition}

\begin{proof}
The proof follows from direct calculations using the spherical-coordinate representations. 
\end{proof}

\begin{lemma}\label{expansion}\cite[Section 3.3]{CK}
	 $u$ in \eqref{eq:eig} has the following spherical wave expansion in spherical coordinates around the origin, 
	\begin{equation}\label{expan}
	u(\mathbf{ x})=4\pi\sum_{n=0}^{\infty}\sum_{m=-n}^{n}\bsi^na_n^mj_n(\sqrt{\lambda}r)Y_n^m(\theta, \phi),
	\end{equation}
	where $Y_n^m(\theta,\phi)$ is the spherical harmonics that is given by 
	\begin{equation*}\label{sphe harmonic}
	Y_n^m(\theta,\phi)=\sqrt{\frac{2n+1}{4\pi}\frac{(n-|m|)!}{(n+|m|)!}}P_n^{|m|}(\cos\theta)e^{\bsi m\phi}
	\end{equation*}
	with the associated Legendre function $P_n^m(t)$, and $j_n(t)$ is the spherical Bessel function of order $n$.
\end{lemma}


\begin{lemma}\cite[Theorem 2.4.4]{Ned}\label{base2}
	Let $0\leq m,l\leq n$. In the spherical coordinate system, the associated Legendre functions fulfill the following orthogonality condition for a fixed $n \in \mathbb N$:
    \begin{equation*}\label{ortho3}
    \int_{-\pi}^{\pi}\frac{P_n^m(\cos\theta)P_n^l(\cos\theta)}{\sin\theta}\,d\theta=
    \left\{  
    \begin{array}{cc}
    0&\mbox{ if }\quad l\neq m\\
    \frac{(n+m)!}{m(n-m)!}&\mbox{ if }\quad l=m
    \end{array}
    \right. .
    \end{equation*}
\end{lemma}


\begin{lemma}\label{coeff0}
	Suppose that for $t\in(0, h)$, $h\in\mathbb{R}_+$, 
	\begin{equation}\label{coef1}
	\sum_{n=0}^{\infty}\alpha_nj_n(t)=0,
	\end{equation}
	where $j_n(t)$ is the $n$-th spherical Bessel function. Then
	\begin{equation}\label{coef2}
	\alpha_n=0,\quad n=0,1,2,\cdots.
	\end{equation}
\end{lemma}

\begin{proof}
	By \cite[Section 2.4]{CK} we know that
	\begin{align}\label{j_n}
	j_n(t):=\sum_{p=0}^{\infty}\frac{(-1)^pt^{n+2p}}{2^pp!1\cdot3\cdots(2n+2p+1)}=\frac{t^n}{(2n+1)!!}\Big(1+\sum_{p=1}^{\infty}\frac{(-1)^pt^{2p}}{2^pp!N_{l,n}}\Big)
	\end{align}
	where $N_{l,n}=(2n+3)\cdot(2n+5)\cdots(2n+2p+1)$. Substituting \eqref{j_n} into \eqref{coef1} and comparing the coefficient of $t^n$ ($n=1,2,\cdots$), we can deduce \eqref{coef2}.
\end{proof}

We are in a position to study the general vanishing orders with the help of spherical wave expansion of the Laplacian eigenfunction $u$ in \eqref{eq:eig} around an intersecting edge corner. 

\begin{lemma}\label{an0}
	Let $u$ be a Laplacian eigenfunction to \eqref{eq:eig}. Suppose that there exits an edge corner ${\mathcal E}(\Pi_1, \Pi_2,\bsl)$ such that
	$$
	{\mathcal E}(\Pi_1, \Pi_2,\bsl)\Subset \Omega,
	$$
 where $\Pi_\ell$, $\ell=1,2$, is from either one of $\mathcal{N}^\lambda_\Omega$, $\mathcal{S}^\lambda_\Omega$ or $\mathcal{M}^\lambda_\Omega$. If there exits a sufficiently small $\varepsilon\in\mathbb{R}_+$ such that
	\begin{equation}\label{lem:27 cond1}
		u|_{ B_\varepsilon(\mathbf{0}) \cap \bsl}=0,
	\end{equation}
	then we have 
	\begin{equation}\label{eq:an0}
		a_n^0=0,\quad n \in \mathbb{N} \cup \{0\}
	\end{equation}
	where $a_n^0$ is given in \eqref{expan}. 
\end{lemma}
\begin{proof}
Since  the line segment  $\bsl$  associated with ${\mathcal E}(\Pi_1, \Pi_2,\bsl)$ coincides with the $x_3$-axis, in the spherical coordinate system \eqref{eq:x sph}, for $\bfx \in \bsl$ we know that $\theta=0\mbox{ or }\pi$. Combining with Lemma \ref{expansion}, under the condition \eqref{lem:27 cond1},  we have
\begin{equation}\label{u-1}
u|_{ B_\varepsilon(\mathbf{0}) \cap \bsl}=4\pi\sum_{n=0}^{\infty}\sum_{m=-n}^{n}\bsi^na_n^mj_n(\sqrt{\lambda}r)\sqrt{\frac{2n+1}{4\pi}}\sqrt{\frac{(n-|m|)!}{(n+|m|)!}}P_n^{|m|}(\pm1)e^{\bsi m\phi}=0.
\end{equation}
On the other hand, we have that for $m\in \mathbb N$ (cf. \cite{WB}),
\begin{equation}\label{P}
P_n^{-m}=(-1)^m\frac{(n-m)!}{(n+m)!}P_n^m,\ P_n^m(\pm1)=0,\ P_n^0(+1)=1,\
P_n^0(-1)=(-1)^n.
\end{equation}
Substituting \eqref{P} into \eqref{u-1}, it is easy to see that
\begin{equation*}\label{simp u-1}
\sum_{n=0}^{\infty}\bsi^n\sqrt{\frac{2n+1}{4\pi}}a_n^0j_n(\sqrt{\lambda}r)=0.
\end{equation*}
By virtue of Lemma~\ref{coeff0}, we readily have 
\begin{equation}\notag
\bsi^n\sqrt{\frac{2n+1}{4\pi}}a_n^0=0\quad\mbox{ for }n=0,1,2,\cdots,
\end{equation}
which completes the proof.	
\end{proof}

First, we consider the case that two nodal planes intersect each other to yield the edge corner. 
 
\begin{theorem}\label{soundsoft2}
	Let $u$ be a Laplacian eigenfunction to \eqref{eq:eig}. Consider an edge-corner ${\mathcal E}(\Pi_1, \Pi_2,\bsl)\Subset \Omega$
	where the two planes $\Pi_\ell$, $\ell=1,2$ are assumed to be nodal, namely $\Pi_\ell \in \mathcal{N}^\lambda_\Omega (\ell=1,2)$. 
	{If 
			\begin{equation}\notag
			\angle(\Pi_{1}, \Pi_2)=\phi=\alpha\cdot\pi,\quad \alpha\in(0,1),
			\end{equation}
	signifying the corresponding dihedral angle.}
	where $\alpha$ satisfies for an $N\in\mathbb{N}$, $N\geq 3$,  
	\begin{equation}
         \alpha\neq \frac{q}{p}, \ \  p=1,2,\cdots, N-1, \  q=1,2,\cdots, p-1,
	\end{equation}
	then $u$ vanishes up to order at least $N$ at the edge-corner point $\mathbf 0$. 
\end{theorem}

\begin{proof}
	Since $u|_{\Pi_i}\equiv0$, $i=1,2$, we have the following two equations by Lemma \ref{expansion}:
	\begin{align}
	u|_{\Pi_1}=& 4\pi\sum_{n=0}^{\infty}\sum_{m=-n}^{n}\bsi^na_n^mj_n(\sqrt{\lambda}r)\sqrt{\frac{2n+1}{4\pi}}\sqrt{\frac{(n-|m|)!}{(n+|m|)!}}P_n^{|m|}(\cos\theta)=0,\label{2nodal1}\\
	u|_{\Pi_2}=& 4\pi\sum_{n=0}^{\infty}\sum_{m=-n}^{n}\bsi^na_n^mj_n(\sqrt{\lambda}r)\sqrt{\frac{2n+1}{4\pi}}\sqrt{\frac{(n-|m|)!}{(n+|m|)!}}P_n^{|m|}(\cos\theta)e^{\bsi m\alpha\cdot\pi}=0,\label{2nodal2}
	\end{align}
	where $\phi=0$ on $\Pi_1$ and $\phi=\alpha\cdot\pi, \alpha\in(0,1)$ on $\Pi_2$. It is obvious that $u|_{\bsl}=0$, then from Lemma \ref{an0}, we have that \eqref{eq:an0} holds. Thus comparing the coefficient of $r$ and substituting $a_n^0=0$, $n=0,1$ into \eqref{2nodal1} and \eqref{2nodal2}, we have
    \begin{equation*}\label{2nodal1-sim}
    (a_1^1+a_1^{-1})P_1^1(\cos\theta)=0,\ \
    (a_1^1e^{\bsi\alpha\cdot\pi}+a_1^{-1}e^{-\bsi\alpha\cdot\pi})P_1^1(\cos\theta)=0.   	
    \end{equation*}
    Since $\theta\in(0, \pi)$ is arbitrary, utilizing the orthogonality condition (Lemma \ref{base2}), we can deduce that 
    \begin{equation*}\label{a11}
    a_1^1+a_1^{-1}=0,\ \
    a_1^1e^{\bsi\alpha\cdot\pi}+a_1^{-1}e^{-\bsi\alpha\cdot\pi}=0.
    \end{equation*}
    Therefore, if $\alpha\neq0, 1$, we can obtain that $a_1^{\pm1}=0$. 
    
    Assume that $a_{n-1}^m=0$, $m=\pm1,\pm2,\cdots,\pm(n-1)$. We next show by induction that $a_n^m=0$, $m=\pm1,\pm2,\cdots,\pm n$. Indeed, comparing the coefficient of $r^n$, we have
    \begin{align}
        \sum_{m=-n}^{n}\bsi^na_n^m\frac{\sqrt{\lambda}^n}{(2n+1)!!}\sqrt{\frac{2n+1}{4\pi}}\sqrt{\frac{(n-|m|)!}{(n+|m|)!}}P_n^{|m|}(\cos\theta)=&0,\label{2nodaln1}\\
         \sum_{m=-n}^{n}\bsi^na_n^m\frac{\sqrt{\lambda}^n}{(2n+1)!!}\sqrt{\frac{2n+1}{4\pi}}\sqrt{\frac{(n-|m|)!}{(n+|m|)!}}P_n^{|m|}(\cos\theta )e^{\bsi m\alpha\cdot\pi}=&0.\label{2nodaln2}
    \end{align}
    Similarly, substituting $a_n^0=0$ into \eqref{2nodaln1} and \eqref{2nodaln2}, noting $\theta$ is arbitrary, and utilizing the orthogonality condition (Lemma \ref{base2}) again we can derive that, for $m=1,2,\cdots$,
    \begin{equation}\notag
    a_n^m+a_n^{-m}=0,\ \
    a_n^me^{\bsi m\alpha\cdot\pi}+a_n^{-m}e^{-\bsi m\alpha\cdot\pi}=0.
    \end{equation}
    Hence if $\alpha\neq\frac{k}{m}$, $k=1,2,\cdots,m-1$, the coefficient matrix fulfills
    \begin{equation}\notag
    \left|\begin{array}{cc} 
    1 &  1 \\ 
    e^{\bsi m\alpha\cdot\pi} & e^{-\bsi m\alpha\cdot\pi}
    \end{array}\right| 
    =-2\bsi\sin m\alpha\cdot\pi\neq0,
    \end{equation}
    which yields that $a_n^m=0$ for $m=\pm1,\pm2,\cdots,\pm n$.
    
    The proof is complete.
\end{proof}

We proceed to consider the case that a nodal plane $\Pi_1\in\mathcal{N}^\lambda_\Omega$ intersects with a generalized singular plane $\Pi_2\in\mathcal{M}^\lambda_\Omega$.

\begin{theorem}\label{gene-nodal}
	Let $u$ be a Laplacian eigenfunction to \eqref{eq:eig}. Consider an edge corner ${\mathcal E}(\Pi_1, \Pi_2,\bsl)\Subset \Omega$ with
	{\color{black}{
	$$
	\Pi_1 \in \mathcal{N}^\lambda_\Omega, \quad\Pi_2 \in \mathcal{M}_\Omega^\lambda \quad\mbox{and}\quad \angle(\Pi_{1},\Pi_2)=\phi=\alpha\cdot\pi, \quad\alpha\in(0,1). 
	$$}}
	If for an $N\in\mathbb{N}$, $N\geq 2$, there holds
			\begin{equation}\notag
		\alpha\neq \frac{2q+1}{2p},\ \ p=1,2,\cdots, N-1,\ q=1,2,\cdots, p-1 ,
		\end{equation}
	then $u$ vanishes up to order at least $N$ at the edge-corner point $\mathbf 0$. 
\end{theorem}

\begin{proof}
	Since $u|_{\Pi_1}\equiv0$, it is direct to know that $u|_{\bsl}\equiv0$ which indicates that $a_n^0=0$ for $n=0,1,2,\cdots$ from Lemma \ref{an0}. Furthermore, by Lemma \ref{expansion} we have
	\begin{equation}\label{mix pi1}
	u|_{\Pi_1}=4\pi\sum_{n=0}^{\infty}\sum_{m=-n}^{n}\bsi^na_n^mj_n(\sqrt{\lambda}r)\sqrt{\frac{2n+1}{4\pi}}\sqrt{\frac{(n-|m|)!}{(n+|m|)!}}P_n^{|m|}(\cos\theta)=0.
	\end{equation}
	Combining with Proposition \ref{neumann bd}, we have the following expression on $\Pi_2$:
	\begin{align}\label{mix pi2}
	&\frac{\partial u}{\partial\nu}+\eta u\Big|_{\Pi_2}
	=\frac{1}{r\sin\theta}\frac{\partial u}{\partial \phi}+\eta u\Big|_{\phi=\alpha\cdot\pi}\notag\\
	=&\frac{1}{r\sin\theta}4\pi\sum_{n=0}^{\infty}\sum_{m=-n}^{n}\bsi^{n+1}ma_n^mj_n(\sqrt{\lambda}r)\sqrt{\frac{2n+1}{4\pi}}\sqrt{\frac{(n-|m|)!}{(n+|m|)!}}P_n^{|m|}(\cos\theta)e^{\bsi m\alpha\cdot\pi}\notag\\
	+&\eta\cdot4\pi\sum_{n=0}^{\infty}\sum_{m=-n}^{n}\bsi^{n}a_n^mj_n(\sqrt{\lambda}r)\sqrt{\frac{2n+1}{4\pi}}\sqrt{\frac{(n-|m|)!}{(n+|m|)!}}P_n^{|m|}(\cos\theta)e^{\bsi m\alpha\cdot\pi}=0.
	\end{align}
	Since $\theta\in(0, \pi)$ and $r>0$, multiplying $r\sin\theta$ on the both sides of \eqref{mix pi2} we can obtain that 
	\begin{align}\label{mix pi2 sim}
	&\sum_{n=0}^{\infty}\sum_{m=-n}^{n}\bsi^{n+1}ma_n^mj_n(\sqrt{\lambda}r)\sqrt{\frac{2n+1}{4\pi}}\sqrt{\frac{(n-|m|)!}{(n+|m|)!}}P_n^{|m|}(\cos\theta)e^{\bsi m\alpha\cdot\pi}\notag\\
	&+\eta\cdot r\sin\theta\sum_{n=0}^{\infty}\sum_{m=-n}^{n}\bsi^{n}a_n^mj_n(\sqrt{\lambda}r)\sqrt{\frac{2n+1}{4\pi}}\sqrt{\frac{(n-|m|)!}{(n+|m|)!}}P_n^{|m|}(\cos\theta)e^{\bsi m\alpha\cdot\pi}=0.
	\end{align} 
	Following a similar argument to Theorem \ref{soundsoft2}, we compare the coefficient of $r$ in \eqref{mix pi1} and \eqref{mix pi2 sim} respectively. For \eqref{mix pi1} we know that
	\begin{equation*}\label{mix r1}
	\sum_{m=-1}^{1}\bsi a_1^m\frac{\sqrt{\lambda}}{3!!}\sqrt{\frac{3}{4\pi}}\sqrt{\frac{(1-|m|)!}{(1+|m|)!}}P_1^{|m|}(\cos\theta)=0.
	\end{equation*} 
	Since $a_1^0=0$, using Lemma \ref{base2} we can deduce that 
	\begin{equation}\label{mixa11}
	a_1^1+a_1^{-1}=0.
	\end{equation}
	For \eqref{mix pi2 sim}, we have 
	\begin{equation}\label{mixa112}
	\sum_{m=-1}^{1}ma_1^m\frac{\sqrt{\lambda}}{3!!}\sqrt{\frac{3}{4\pi}}\sqrt{\frac{(1-|m|)!}{(1+|m|)!}}P_1^{m}(\cos\theta)e^{\bsi m\alpha\cdot\pi}=0
	\end{equation}
	since $a_0^0=0$. By the orthogonality condition of $P_1^m$ for arbitrary $\theta\in(0, \pi)$ and the fact that $a_1^0=0$ we can simplify \eqref{mixa112} as 
	\begin{equation*}
	a_1^1e^{\bsi\alpha\cdot\pi}-a_1^{-1}e^{-\bsi\alpha\cdot\pi}=0.
	\end{equation*}
	Combining \eqref{mixa11} with \eqref{mixa112} we can obtain that if $\alpha\neq\frac{1}{2}$, then $a_1^{\pm1}=0$. By induction, we assume that $a_{n-1}^m=0$, $m=\pm1,\pm2,\cdots,\pm(n-1)$. Considering the coefficient of $r^n$ in \eqref{mix pi1}, we know that
	\begin{equation*}\label{mixan11}
	\sum_{m=-n}^{n}\bsi^na_n^m\frac{\sqrt{\lambda}^n}{(2n+1)!!}\sqrt{\frac{2n+1}{4\pi}}\sqrt{\frac{(n-|m|)!}{(n+|m|)!}}P_n^{|m|}(\cos\theta)=0,
	\end{equation*}
	in which we can derive 
	\begin{equation}\label{mixan12}
	a_n^m+a_n^{-m}=0 \quad\mbox{ for }m=1,2,\cdots
	\end{equation}
	by virtue of the fact that $a_n^0=0$ and Lemma \ref{base2}. Similarly, for \eqref{mix pi2 sim}, the coefficient of $r^n$ fulfills that
	\begin{align}\label{mixan21}
&\sum_{m=-n}^{n}\bsi^{n+1}ma_n^m\frac{\sqrt{\lambda}^n}{(2n+1)!!}\sqrt{\frac{2n+1}{4\pi}}\sqrt{\frac{(n-|m|)!}{(n+|m|)!}}P_n^{|m|}(\cos\theta)e^{\bsi m\alpha\cdot\pi}\notag\\
&+\eta\cdot \sin\theta\sum_{m=-(n-1)}^{n-1}\bsi^{n-1}a_{n-1}^m\frac{\sqrt{\lambda}^{n-1}}{(2n-1)!!}\sqrt{\frac{2n-1}{4\pi}}\sqrt{\frac{(n-1-|m|)!}{(n-1+|m|)!}}P_{n-1}^{|m|}(\cos\theta)e^{\bsi m\alpha\cdot\pi}=0.
\end{align} 
    Substituting $a_{n-1}^m=0$, $m=\pm1, \pm2,\cdots,\pm(n-1)$, and $a_n^0=0$ into \eqref{mixan21}, utilizing Lemma \ref{base2} again we have
    \begin{equation}\label{mixan22}
    a_n^me^{\bsi m\alpha\cdot\pi}-a_n^{-m}e^{-\bsi m\alpha\cdot\pi}=0.
    \end{equation}
    Therefore, from \eqref{mixan12} and \eqref{mixan22}, we can derive that if $\alpha\neq\frac{2k+1}{2m}$ ($k=0,1,\cdots,m-1$), the coefficient matrix satisfies
    \begin{equation}\notag
    \left|\begin{array}{cc} 
    1 &  1 \\ 
    e^{\bsi m\alpha\cdot\pi} & -e^{-\bsi m\alpha\cdot\pi}
    \end{array}\right| 
    =-2\cos m\alpha\cdot\pi\neq0.
    \end{equation}
    which implies that $a_n^m=0$, $m=\pm1,\pm2,\cdots,\pm n$.
    
    The proof is complete.
\end{proof}

It is straightforward to verify that in Theorem~\ref{gene-nodal}, one can take $\eta\equiv 0$. In such a case, one has

\begin{corollary}\label{sing-nodal}
	Let $u$ be a Laplacian eigenfunction to \eqref{eq:eig}. Consider an edge corner ${\mathcal E}(\Pi_1, \Pi_2,\bsl)\Subset \Omega$ with
	{\color{black}{
	$$
	\Pi_1 \in \mathcal{N}^\lambda_\Omega, \quad\Pi_2 \in \mathcal{S}_\Omega^\lambda \quad\mbox{and}\quad \angle(\Pi_{1},\Pi_2)=\phi=\alpha\cdot\pi, \quad\alpha\in(0,1). 
	$$}}
	If for an $N\in\mathbb{N}$, $N\geq 2$, there holds
			\begin{equation}\notag
		\alpha\neq \frac{2q+1}{2p},\ \ p=1,2,\cdots, N-1,\ q=1,2,\cdots, p-1 ,
		\end{equation}
	then $u$ vanishes up to order at least $N$ at the edge-corner point $\mathbf 0$. 
\end{corollary}

%


\subsection{Vanishing orders at an edge-corner point intersected by generalized singular planes}

In this subsection, we consider the case that an edge corner ${\mathcal E}(\Pi_1, \Pi_2,\bsl)$ is intersected by two generalised singular planes, namely  $ \Pi_\ell  \in \mathcal{M}^\lambda_\Omega$, $\ell=1,2$. In what follows, we signify the boundary parameters on $\Pi_\ell$ to be $\eta_\ell$, $\ell=1,2$. We can derive the following three theorems. 


\begin{theorem}\label{2generalize}
	Let $u$ be a Laplacian eigenfunction to \eqref{eq:eig}. Consider an edge corner ${\mathcal E}(\Pi_1, \Pi_2,\bsl)\Subset\Omega$ with $ \Pi_\ell  \in \mathcal{M}^\lambda_\Omega, \ell=1,2 $ and $\angle(\Pi_{1},\Pi_2)=\phi=\alpha\cdot\pi$ for $\alpha\in(0,1)$. If there exits a sufficiently small radius  $\varepsilon\in \R_+$ such that
	\begin{equation}\label{Thm:211 cond2}
		u|_{ B_\varepsilon(\mathbf{0}) \cap \bsl}\equiv0,
	\end{equation}
	and for an $N\in\mathbb{N}$, $N\geq 3$,
	\begin{equation}\notag
	 \alpha\neq \frac{q}{p},\ \ p=1,2,\cdots,N-1,\ q=1, 2, \cdots, p-1,
	\end{equation}
then $u$ vanishes up to the order at least $N$ at the edge-corner point $\mathbf 0$.
\end{theorem}

\begin{proof}
	Since $u|_{\Pi_i}=\frac{\partial u}{\partial \nu}+\eta_iu=0$, $i=1,2$,  we have by using Proposition \ref{neumann bd} that
    \begin{equation}\notag
    \begin{split}
    \frac{\partial u}{\partial\nu}+\eta_1 u\Big|_{\Pi_1}=-\frac{1}{r\sin\theta}\frac{\partial u}{\partial \phi}+\eta_1 u\Big|_{\phi=0}=0,\\
      \frac{\partial u}{\partial\nu}+\eta_2 u\Big|_{\Pi_2}=\frac{1}{r\sin\theta}\frac{\partial u}{\partial \phi}+\eta_2 u\Big|_{\phi=\alpha\cdot\pi}=0,
      \end{split}
    \end{equation}
    which can be written more explicitly in spherical coordinate system by Lemma \ref{expansion} as:
    \begin{align}\label{impe1}
    &-\sum_{n=0}^{\infty}\sum_{m=-n}^{n}\bsi^{n+1}ma_n^mj_n(\sqrt{\lambda}r)\sqrt{\frac{2n+1}{4\pi}}\sqrt{\frac{(n-|m|)!}{(n+|m|)!}}P_n^{|m|}(\cos\theta)\notag\\
    &+C_1r\sin\theta\sum_{n=0}^{\infty}\sum_{m=-n}^{n}\bsi^{n}a_n^mj_n(\sqrt{\lambda}r)\sqrt{\frac{2n+1}{4\pi}}\sqrt{\frac{(n-|m|)!}{(n+|m|)!}}P_n^{|m|}(\cos\theta)=0,
    \end{align}
    and 
    \begin{align}\label{impe2}
    &\sum_{n=0}^{\infty}\sum_{m=-n}^{n}\bsi^{n+1}ma_n^mj_n(\sqrt{\lambda}r)\sqrt{\frac{2n+1}{4\pi}}\sqrt{\frac{(n-|m|)!}{(n+|m|)!}}P_n^{|m|}(\cos\theta)e^{\bsi m\alpha\cdot\pi}\notag\\
    &+C_2r\sin\theta\sum_{n=0}^{\infty}\sum_{m=-n}^{n}\bsi^{n}a_n^mj_n(\sqrt{\lambda}r)\sqrt{\frac{2n+1}{4\pi}}\sqrt{\frac{(n-|m|)!}{(n+|m|)!}}P_n^{|m|}(\cos\theta)e^{\bsi m\alpha\cdot\pi}=0.
    \end{align}
    Under the  condition \eqref{Thm:211 cond2}, from Lemma \ref{an0} we know that
  \begin{equation}\label{eq:242an0}
  	a_n^0=0,\quad \mbox{ for } n=0,1,2,\cdots. 
  \end{equation}  Comparing the coefficient of $r^1$ in \eqref{impe1} and \eqref{impe2} respectively we have:
    \begin{equation}\notag
    \begin{split}
    \sum_{m=-1}^{1}ma_1^m\frac{\sqrt{\lambda}}{3!!}\sqrt{\frac{3}{4\pi}}\sqrt{\frac{(1-|m|)!}{(1+|m|)!}}P_1^{|m|}(\cos\theta)=0,\\
        -\sum_{m=-1}^{1}ma_1^m\frac{\sqrt{\lambda}}{3!!}\sqrt{\frac{3}{4\pi}}\sqrt{\frac{(1-|m|)!}{(1+|m|)!}}P_1^{|m|}(\cos\theta)e^{\bsi m\alpha\cdot\pi}=0.
    \end{split}
    \end{equation}
    Utilizing the orthogonality condition (Lemma \ref{base2}) and the fact that $a_1^0=0$ we can obtain the linear system with respect to $a_1^{\pm1}$ as 
    \begin{equation*}
    a_1^1-a_1^{-1}=0,\ \ a_1^1e^{\bsi \alpha\cdot\pi}-a_1^{-1}e^{-\bsi \alpha\cdot\pi}=0.
    \end{equation*}
    Since $\alpha\in(0, 1)$, which indicates that $\phi\neq0, \pi$, it is easy to see that $a_1^{\pm1}=0$. Using the same argument, by induction, we assume that
    \begin{equation}\label{eq:244anm}
    	a_{n-1}^m=0, \quad m=\pm1,\pm2,\cdots,\pm(n-1).
    \end{equation}
     Then by considering the coefficient of $r^n$ in \eqref{impe1} and \eqref{impe2} we have
        \begin{align}\label{impe1-1}
    &-\sum_{m=-n}^{n}\bsi^{n+1}ma_n^m\frac{\sqrt{\lambda}^n}{(2n+1)!!}\sqrt{\frac{2n+1}{4\pi}}\sqrt{\frac{(n-|m|)!}{(n+|m|)!}}P_n^{|m|}(\cos\theta)\notag\\
    &+C_1\sin\theta\sum_{m=-(n-1)}^{n-1}\bsi^{n-1}a_{n-1}^m\frac{\sqrt{\lambda}^{n-1}}{(2n-1)!!}\sqrt{\frac{2n-1}{4\pi}}\sqrt{\frac{(n-1-|m|)!}{(n-1+|m|)!}}P_{n-1}^{|m|}(\cos\theta)=0,
    \end{align}
    and 
    \begin{align}\label{impe2-1}
    &\sum_{m=-n}^{n}\bsi^{n+1}ma_n^m\frac{\sqrt{\lambda}^n}{(2n+1)!!}\sqrt{\frac{2n+1}{4\pi}}\sqrt{\frac{(n-|m|)!}{(n+|m|)!}}P_n^{|m|}(\cos\theta)e^{\bsi m\alpha\cdot\pi}\notag\\
    &+C_2\sin\theta\sum_{m=-(n-1)}^{n-1}\bsi^{n-1}a_{n-1}^m\frac{\sqrt{\lambda}^{n-1}}{(2n-1)!!}\sqrt{\frac{2n-1}{4\pi}}\sqrt{\frac{(n-1-|m|)!}{(n-1+|m|)!}}P_{n-1}^{|m|}(\cos\theta)e^{\bsi m\alpha\cdot\pi}=0.
    \end{align}
  By induction,   substituting \eqref{eq:242an0} and \eqref{eq:244anm}
    into \eqref{impe1-1} and \eqref{impe2-1}, using Lemma \ref{base2} we can deduce that for $m\in\mathbb{N}_+$,
    \begin{equation}\label{impe-anm}
    \begin{cases}
    &a_n^m-a_n^{-m}=0,\\
    &a_n^me^{\bsi m\alpha\cdot\pi}-a_n^{-m}e^{-\bsi m\alpha\cdot\pi}=0.
    \end{cases}
    \end{equation}
    Hence if $\alpha\neq\frac{k}{m}$, $k=1,2,\cdots,m-1$, the coefficient matrix of \eqref{impe-anm} fulfills 
    \begin{equation}\notag
    \left|\begin{array}{cc} 
    1 &  -1 \\ 
    e^{\bsi m\alpha\cdot\pi} & -e^{-\bsi m\alpha\cdot\pi}
    \end{array}\right| 
    =2\bsi\sin m\alpha\cdot\pi\neq0.
    \end{equation}
    which implies that $a_n^m=0$, $m=\pm1,\pm2,\cdots,\pm n$.
    
    The proof is complete.
\end{proof}

\begin{remark}\label{nec condi1}
	It is very important and necessary to assume that $u|_{ B_\varepsilon(\mathbf{0}) \cap \bsl}\equiv0$ in Theorem \ref{2generalize}, where we use this condition to obtain $a_n^0=0$, $n=0,1,2,\cdots$. Without this assumption we can not derive the recursive equations with respect to $a_n^m$ from \eqref{impe1} and \eqref{impe2} to obtain the vanishing results.
\end{remark}

\begin{remark}\label{rem:gg1}
It is straightforward to verify that in the proof of Theorem~\ref{2generalize}, $\eta_1$ and/or $\eta_2$ can be zero. That is, Theorem~\ref{2generalize} also includes the case that at least one of the two planes $\Pi_\ell$ is a singular plane. We choose not to present those results in order to avoid repetition. 
\end{remark}

\section{Vanishing orders at vertex-corner points}\label{sec3}

In this section, we study the vanishing property of the Laplacian eigenfunction to \eqref{eq:eig} at a vertex-corner point associated with a vertex corner $\mathcal{V}(\{\Pi_\ell\}_{\ell=1}^n, \bfx_0)\Subset\Omega$, where $\Pi_\ell$ could be either a nodal plane, or a singular plane or a generalized singular plane. It is known that an edge corner  ${\mathcal E}(\Pi_1, \Pi_2,\bsl) $ can be regarded as a part of a vertex corner $\mathcal{V}(\{\Pi_\ell\}_{\ell=1}^n, \bfx_0)$. In Section \ref{sec2}, we have unveiled that the vanishing order of the eigenfunction $u$ at an edge-corner point can be determined by the the intersecting dihedral angle of ${\mathcal E}(\Pi_1, \Pi_2,\bsl) $ under some generic condition \eqref{Thm:211 cond2}. However, in this section, we concentrate on another condition
\begin{equation}\label{eq:cond sec}
	u(\bfx_0)=0,
\end{equation}
to study the vanishing property of $u$ at $\bfx_0$. We should point out that \eqref{eq:cond sec} is much more relaxed compared to \eqref{Thm:211 cond2}, and it can be easily fulfilled in certain generic case, e.g, superpositions of four eigenfunctions at the point $\bfx_0$. In particular, such a condition \eqref{eq:cond sec} can be used to show the unique determination of some polyhedral obstacles in $\R^3$ by finitely many measurements in the inverse obstacle scattering problem. We shall give more detailed discussions in Section \ref{sec5}.

Similar to Section \ref{sec2}, without loss of generality, we assume that the vertex-corner point $\bfx_0$ of $\mathcal{V}(\{\Pi_\ell\}_{\ell=1}^n, \bfx_0)$ coincides with the origin. We first focus on the case that $n=3$, which implies that the vertex corner $\mathcal{V}(\{\Pi_\ell\}_{\ell=1}^3, \bfx_0)$ is formed by three planes;  see Figure \ref{3d vertex corner} for a schematic illustration. For $n>3$, the related results can be derived in a similar way; see Theorem \ref{3d vertex-n plane}--\ref{3d vertex-n plane22}. It is obvious that $\mathcal{V}(\{\Pi_\ell\}_{\ell=1}^3, \bfx_0)$ is formed by three edge corners $\mathcal{E}(\Pi_1,\Pi_2, \bsl_1 )$, $\mathcal{E}(\Pi_2,\Pi_3, \bsl_2 )$  and $\mathcal{E}(\Pi_3,\Pi_1, \bsl_3 )$ where $\bsl_1,\, \bsl_2$ and  $\bsl_3$ are three line segments of $\Pi_1\cap \Pi_2$, $\Pi_2\cap \Pi_3$ and $\Pi_3\cap \Pi_1$ respectively. Hence, if either one of the three planes $\Pi_\ell$ is nodal, say $\Pi_3$, then one can apply the results in Section 2 to the edge corners $\mathcal{E}(\Pi_2,\Pi_3, \bsl_2 )$  and $\mathcal{E}(\Pi_3,\Pi_1, \bsl_3 )$ to derive a certain vanishing order at the vertex-corder point, by regarding it as an edge-corner point associated with $\mathcal{E}(\Pi_2,\Pi_3, \bsl_2 )$  and $\mathcal{E}(\Pi_3,\Pi_1, \bsl_3 )$, respectively. Hence, in such a case, we shall mainly focus on the vanishing order generated through the intersection the two planes $\Pi_1$ and $\Pi_2$, both of which are assumed not to be nodal.


\begin{figure}[htbp]\label{3d vertex corner}
	\centering
	\includegraphics[width=0.3\linewidth]{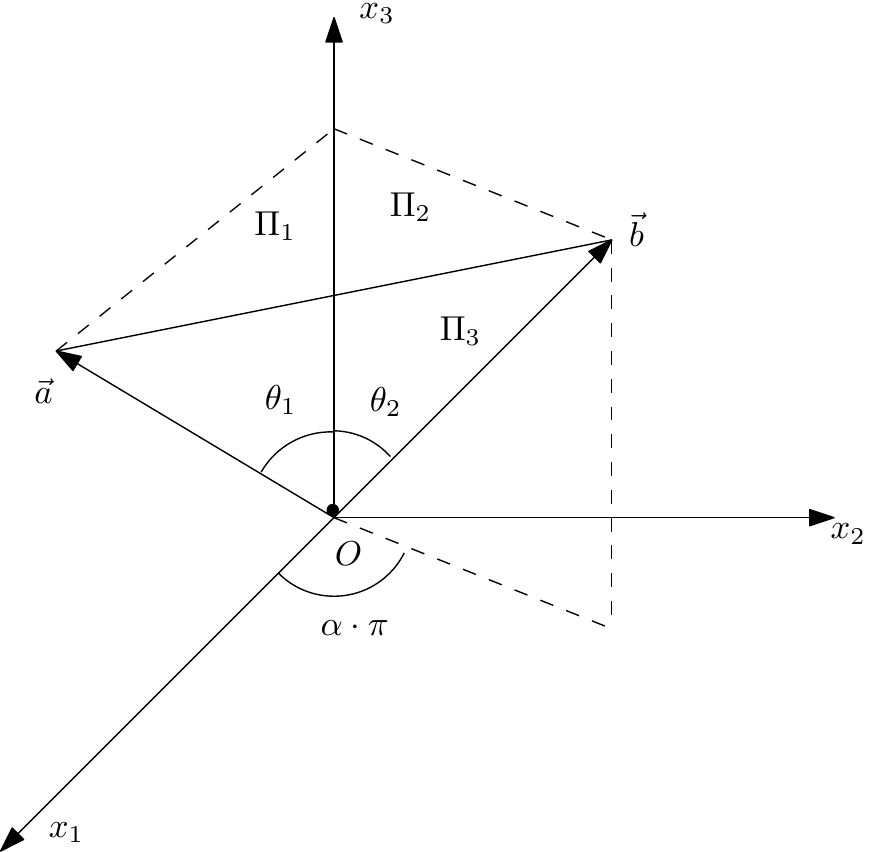}
	\caption{Schematic illustration of a vertex corner that is intersected by $\Pi_1$, $\Pi_2$ and $\Pi_3$.}
	\label{fig:3d-vertex2}
\end{figure}

\begin{theorem}\label{3d vertex-thm1}
	Let $u$ be a Laplacian eigenfunction to \eqref{eq:eig}. Consider a vertex corner ${\mathcal V}(\{\Pi_\ell\}_{\ell=1}^3,{\mathbf 0})\Subset \Omega$ with $ \Pi_\ell  \in \mathcal{M}^\lambda_\Omega $, $\ell=1,2$, $\angle(\Pi_{1},\Pi_2)=\phi=\alpha\cdot\pi$, $\alpha\in(0,1)$ and  $\Pi_3\in\mathcal{N}_\Omega^\lambda$. Assume that $\Pi_3=\mathrm{span}\{\vec{a},\vec{b}\}$, where $\vec{a}=(r, \theta_1, 0)\in\Pi_1\cap\Pi_3$ and $\vec{b}=(r, \theta_2, \alpha\cdot\pi)\in\Pi_2\cap\Pi_3$ for $r>0$, $\alpha\in(0,1)$, and fixed $\theta_1$ and $\theta_2$ in the spherical coordinate system. If for an $N\in\mathbb{N}$, $N\geq 3$, there holds 
\begin{equation}\notag
 P_p^0(\cos\theta_i)\neq0, i=1\mbox{ or }2,\mbox{ and }\alpha\neq \frac{q}{p}, p=1,2,\cdots,N-1,\ q=1, 2,\cdots, p-1,
\end{equation}
where $P_p^0$ is the Legendre polynomial, then the vanishing order of $u$ at $\mathbf 0$ generated by the intersection of the two planes $\Pi_1$ and $\Pi_2$ is at least up to order $N$. 
\end{theorem}

\begin{proof}
	Since $\Pi_1$ and $\Pi_2$ are two generalized singular planes, we have 
	\begin{equation}\label{thm1-1}
	\frac{\partial u}{\partial\nu}+\eta_1u\Big|_{\Pi_1}=0\quad\mbox{ and }\quad\frac{\partial u}{\partial\nu}+\eta_2u\Big|_{\Pi_2}=0.
	\end{equation}
	By Proposition \ref{neumann bd} and Lemma \ref{expansion}, we can write \eqref{thm1-1} explicitly as 
	\begin{align}\label{thm1-2-1}
	&-\frac{1}{r\sin\theta}\frac{\partial u}{\partial \phi}+\eta_1u\Big|_{\phi=0}\notag\\
	&=-\frac{1}{r\sin\theta}4\pi\sum_{n=0}^{\infty}\sum_{m=-n}^{n}\bsi^{n+1}ma_n^mj_n(\sqrt{\lambda}r)\sqrt{\frac{2n+1}{4\pi}}\sqrt{\frac{(n-|m|)!}{(n+|m|)!}}P_n^{|m|}(\cos\theta)\notag\\
	&+\eta_14\pi\sum_{n=0}^{\infty}\sum_{m=-n}^{n}\bsi^{n}a_n^mj_n(\sqrt{\lambda}r)\sqrt{\frac{2n+1}{4\pi}}\sqrt{\frac{(n-|m|)!}{(n+|m|)!}}P_n^{|m|}(\cos\theta)=0,
	\end{align} 
	and 
	\begin{align}\label{thm1-2-2}
	&\frac{1}{r\sin\theta}\frac{\partial u}{\partial \phi}+\eta_2u\Big|_{\phi=\alpha\cdot\pi}\notag\\
	&=\frac{1}{r\sin\theta}4\pi\sum_{n=0}^{\infty}\sum_{m=-n}^{n}\bsi^{n+1}ma_n^mj_n(\sqrt{\lambda}r)\sqrt{\frac{2n+1}{4\pi}}\sqrt{\frac{(n-|m|)!}{(n+|m|)!}}P_n^{|m|}(\cos\theta)e^{\bsi m \alpha\cdot\pi}\notag\\
	&+\eta_24\pi\sum_{n=0}^{\infty}\sum_{m=-n}^{n}\bsi^{n}a_n^mj_n(\sqrt{\lambda}r)\sqrt{\frac{2n+1}{4\pi}}\sqrt{\frac{(n-|m|)!}{(n+|m|)!}}P_n^{|m|}(\cos\theta)e^{\bsi m\alpha\cdot\pi}=0.
	\end{align} 
	Since $\Pi_3=\mbox{span}\{\vec{a},\vec{b}\}$, where $\vec{a}=(r,\theta_1,0)\in\Pi_1\cap\Pi_3$ and $\vec{b}=(r,\theta_2,\alpha\cdot\pi)\in\Pi_2\cap\Pi_3$ for fixed $\theta_1$, $\theta_2$ and $u|_{\Pi_3}\equiv0$. It is direct to see  $u|_{\vec{a}}=u|_{\vec{b}}=0$, which further indicates that
	\begin{equation}\label{thm-1-3}
	u|_{\vec{a}}=4\pi\sum_{n=0}^{\infty}\sum_{m=-n}^{n}\bsi^{n}a_n^mj_n(\sqrt{\lambda}r)\sqrt{\frac{2n+1}{4\pi}}\sqrt{\frac{(n-|m|)!}{(n+|m|)!}}P_n^{|m|}(\cos\theta_1)=0,
	\end{equation}
	and 
	\begin{equation}\label{thm-1-3-2}
	u|_{\vec{b}}=4\pi\sum_{n=0}^{\infty}\sum_{m=-n}^{n}\bsi^{n}a_n^mj_n(\sqrt{\lambda}r)\sqrt{\frac{2n+1}{4\pi}}\sqrt{\frac{(n-|m|)!}{(n+|m|)!}}P_n^{|m|}(\cos\theta_2)e^{\bsi m\alpha\cdot\pi}=0.
	\end{equation}
	
	Combining with \eqref{thm1-2-1} and \eqref{thm1-2-2}, it suffices to use \eqref{thm-1-3} or \eqref{thm-1-3-2} to study the coefficients of $r^n$, $n\in\mathbb{N}$. In what follows, without loss of generality, we discuss \eqref{thm-1-3} for instance. Since $u|_{\vec{a}}\equiv0$, the coefficient of $r^0$ fulfills that 
	\begin{equation}\notag
	4\pi a_0^0\sqrt{\frac{1}{4\pi}}P_0^0(\cos\theta_1)=0,
	\end{equation} 
	where we can know that $a_0^0=0$ since $P_0^0\equiv1$.
	Consider the coefficient of $r$, from \eqref{thm1-2-1}, \eqref{thm1-2-2} and \eqref{thm-1-3}, we can respectively see that
	\begin{equation}\label{thm1-r1}
	\sum_{m=-1}^{1}ma_1^m\frac{\sqrt{\lambda}}{3!!}\sqrt{\frac{3}{4\pi}}\sqrt{\frac{(1-|m|)!}{(1+|m|)!}}P_1^{|m|}(\cos\theta)+C_1\sin\theta a_0^0\sqrt{\frac{1}{4\pi}}P_0^0(\cos\theta)=0,
	\end{equation}
	\begin{equation}\label{thm1-r12}
	\sum_{m=-1}^{1}ma_1^m\frac{\sqrt{\lambda}}{3!!}\sqrt{\frac{3}{4\pi}}\sqrt{\frac{(1-|m|)!}{(1+|m|)!}}P_1^{|m|}(\cos\theta)e^{\bsi m\alpha\cdot\pi}
	-C_2\sin\theta\sum_{m=-1}^{1}a_0^0\sqrt{\frac{1}{4\pi}}P_0^0(\cos\theta)e^{\bsi m\alpha\cdot\pi}=0,
	\end{equation}
	\begin{equation}\label{thm1-r13}
	\sum_{m=-1}^{1}\bsi a_1^m\frac{\sqrt{\lambda}}{3!!}\sqrt{\frac{3}{4\pi}}\sqrt{\frac{(1-|m|)!}{(1+|m|)!}}P_1^{|m|}(\cos\theta_1)=0.
	\end{equation}
	Substituting $a_0^0=0$ into \eqref{thm1-r1} and \eqref{thm1-r12}, combining with Lemma \ref{base2}, we can directly derive the following linear system with respect to $a_1^{\pm1}$:
	\begin{equation}\label{thm1-a1}
	a_1^1-a_1^{-1}=0,\ \
	a_1^1e^{\bsi\alpha\cdot\pi}-a_1^{-1}e^{-\bsi\alpha\cdot\pi}=0.
	\end{equation}
	Thus we know that $a_1^{\pm1}=0$ since $\alpha\in(0,1)$. As a consequence, in \eqref{thm1-r13}, if $P_1^0(\cos\theta_1)\neq0$, we can deduce that $a_1^0=0$ easily. 
	
	By induction, we assume that $a_{n-1}^m=0$ for $m=0,\pm1,\pm2,\cdots,\pm(n-1)$. Then considering the coefficient of $r^n$, by \eqref{thm1-2-1}, \eqref{thm1-2-2} and \eqref{thm-1-3}, we have
	\begin{align}\label{thm1-rn1}
	&-\sum_{m=-n}^{n}\bsi^{n+1}ma_n^m\frac{\sqrt{\lambda}^n}{(2n+1)!!}\sqrt{\frac{2n+1}{4\pi}}\sqrt{\frac{(n-|m|)!}{(n+|m|)!}}P_n^{|m|}(\cos\theta)\notag\\
	&+C_1\sin\theta\sum_{m=-(n-1)}^{n-1}\bsi^{n-1}a_{n-1}^m\frac{\sqrt{\lambda}^{n-1}}{(2n-1)!!}\sqrt{\frac{2n-1}{4\pi}}\sqrt{\frac{(n-1-|m|)!}{(n-1+|m|)!}}P_{n-1}^{|m|}(\cos\theta)=0,
	\end{align}
	\begin{align}\label{thm1-rn2}
	&\sum_{m=-n}^{n}\bsi^{n+1}ma_n^m\frac{\sqrt{\lambda}^n}{(2n+1)!!}\sqrt{\frac{2n+1}{4\pi}}\sqrt{\frac{(n-|m|)!}{(n+|m|)!}}P_n^{|m|}(\cos\theta)e^{\bsi m\alpha\cdot\pi}\notag\\
    &+C_2\sin\theta\sum_{m=-(n-1)}^{n-1}\bsi^{n-1}a_{n-1}^m\frac{\sqrt{\lambda}^{n-1}}{(2n-1)!!}\sqrt{\frac{2n-1}{4\pi}}\sqrt{\frac{(n-1-|m|)!}{(n-1+|m|)!}}P_{n-1}^{|m|}(\cos\theta)e^{\bsi m \alpha\cdot\pi}=0,
	\end{align}
	and 
	\begin{equation}\label{thm1-rn3}
	\sum_{m=-n}^{n}\bsi^{n}a_n^m\frac{\sqrt{\lambda}^n}{(2n+1)!!}\sqrt{\frac{2n+1}{4\pi}}\sqrt{\frac{(n-|m|)!}{(n+|m|)!}}P_n^{|m|}(\cos\theta_1)=0.
	\end{equation}
	Utilizing the assumption $a_{n-1}^m=0$ for $n=0,\pm1,\pm2,\cdots, \pm(n-1)$ in \eqref{thm1-rn1} and \eqref{thm1-rn2}, from the orthogonality condition in Lemma \ref{base2}, we know that for $m\in\mathbb{N}_+$, $a_n^m$ satisfies
	\begin{equation}\label{thm1-an}
	a_n^m-a_n^{-m}=0,\ \
	a_n^me^{\bsi m \alpha\cdot\pi}-a_n^{-m}e^{-\bsi m\alpha\cdot\pi}.
	\end{equation}
	Therefore, if $\alpha\neq\frac{k}{m}$, $k=1,2,\cdots,m-1$, then the coefficient matrix fulfils
	\begin{equation}\notag
	\left|\begin{array}{cc} 
	1 &  -1 \\ 
	e^{\bsi m\alpha\cdot\pi} & -e^{-\bsi m\alpha\cdot\pi}
	\end{array}\right| 
	=2\bsi\sin m\alpha\cdot\pi\neq0,
	\end{equation}
	and thus $a_n^m=0$ for $m=\pm1,\pm2,\cdots,\pm n$. Now we are in a position to show that $a_n^0=0$. Indeed, substituting $a_n^m=0$, $m=\pm1,\pm2,\cdots,\pm n$ into \eqref{thm1-rn3}, we can obtain that if $P_n^0(\cos\theta_1)\neq0$, then $a_n^0=0$, which completes the proof.	
\end{proof}

In the above proof of Theorem \ref{3d vertex-thm1}, we have analyzed the condition $u|_{\vec{a}}=0$ for illustration. For the condition $u|_{\vec{b}}=0$, we also give the discussion as the following remark.

\begin{remark}\label{vec a}
In the proof of Theorem \ref{3d vertex-thm1}, instead of \eqref{thm-1-3}, if we use \eqref{thm-1-3-2} combining with \eqref{thm1-2-1} and \eqref{thm1-2-2} to consider the coefficient of $r^n$, $n\in\mathbb{N}$, we know that for $r$, \eqref{thm1-r13} becomes 
\begin{equation}\label{rem1-r1}
\sum_{m=-1}^{1}\bsi a_1^m\frac{\sqrt{\lambda}}{3!!}\sqrt{\frac{3}{4\pi}}\sqrt{\frac{(1-|m|)!}{(1+|m|)!}}P_1^{|m|}(\cos\theta_2)e^{\bsi m \alpha\cdot\pi}=0.
\end{equation}
Since we know $a_1^{\pm1}=0$ by \eqref{thm1-r1} and \eqref{thm1-r12}, in \eqref{rem1-r1} we can obtain that if $P_1^0(\cos\theta_2)\neq0$, then $a_1^0=0$. By induction, in order to study $a_n^0$, we replace \eqref{thm1-rn3} by 
\begin{equation}\label{rem1-rn}
\sum_{m=-n}^{n}\bsi^n a_n^m\frac{\sqrt{\lambda}^n}{(2n+1)!!}\sqrt{\frac{2n+1}{4\pi}}\sqrt{\frac{(n-|m|)!}{(n+|m|)!}}P_n^{|m|}(\cos\theta_2)e^{\bsi m \alpha\cdot\pi}=0.
\end{equation}
Substituting $a_n^m=0$, $m=\pm1,\pm2,\cdots,\pm n$, which is derived from \eqref{thm1-an} into \eqref{rem1-rn}, we can deduce that if $P_n^0(\cos\theta_2)\neq0$, then $a_n^0=0$. 

Hence, from the above discussion we know that it is actually equivalent to consider $u|_{\vec{a}}=0$ or $u|_{\vec{b}}=0$ in the proof of Theorem \ref{3d vertex-thm1}. Therefore, in our subsequent study, we shall only prove under the condition with respect to $\vec{a}$.
\end{remark}

In Theorem \ref{3d vertex-thm1}, we have considered the case that $\Pi_3\in\mathcal{N}_\Omega^\lambda$ is a nodal plane. Next, we study a more complicated case that $\Pi_3\in\mathcal{M}_\Omega^\lambda$ is a generalized singular plane.

\begin{theorem}\label{3d vertex-thm2}
Let $u$ be a Laplacian eigenfunction to \eqref{eq:eig}. Consider a vertex corner ${\mathcal V}(\{\Pi_\ell\}_{\ell=1}^3,{\mathbf 0})\Subset \Omega$ with $ \Pi_\ell  \in \mathcal{M}^\lambda_\Omega $, $\ell=1,2,3$ and $\angle(\Pi_{1},\Pi_2)=\phi=\alpha\cdot\pi$, $\alpha\in(0,1)$. Assume that $\Pi_3=\mathrm{span}\{\vec{a},\vec{b}\}$, where $\vec{a}=(r, \theta_1, 0)\in\Pi_1\cap\Pi_3$ and $\vec{b}=(r, \theta_2, \alpha\cdot\pi)\in\Pi_2\cap\Pi_3$ for $r>0$, $\alpha\in(0,1)$, and fixed $\theta_1\in (0, \pi)$ and $\theta_2\in (0, \pi)$ in the spherical coordinate system. If for an $N\in\mathbb{N}$, $N\geq 3$, there holds 
\begin{equation}\notag
 P_p^1(\cos\theta_i)\neq0, i=1\mbox{ or }2,\mbox{ and }\alpha\neq \frac{q}{p}, p=1,2,\cdots,N-1,\ q=1, 2,\cdots, p-1,
\end{equation}
where $P_p^1$ is the Legendre polynomial, then the vanishing order of $u$ at $\mathbf 0$ generated by the intersection of the two planes $\Pi_1$ and $\Pi_2$ is at least up to order $N$. 
\end{theorem}

\begin{proof}
	Since $\Pi_i$, $i=1,2,3$, are three generalized singular planes, we have 
	\begin{equation*}\label{thm2-1}
	\frac{\partial u}{\partial\nu}+\eta_1u\Big|_{\Pi_1}=0,\quad \frac{\partial u}{\partial\nu}+\eta_2u\Big|_{\Pi_2}=0\quad\mbox{and}\quad
	\frac{\partial u}{\partial\nu}+\eta_3u\Big|_{\Pi_3}=0.
	\end{equation*}
	From Theorem \ref{3d vertex-thm1}, we have already known that $u$ satisfies \eqref{thm1-2-1} and \eqref{thm1-2-2} on $\Pi_1$ and $\Pi_2$ respectively. Besides, by Remark \ref{vec a}, we can obtain that 
	\begin{equation}\label{thm2-2-1}
	\frac{\partial u}{\partial \nu}+\eta_3u\Big|_{\vec{a}}=0.
	\end{equation}
	Since $\Pi_3=\mbox{span}\{\vec{a},\vec{b}\}$, which implies that $\nu=\vec{b}\times\vec{a}=(\sin\theta_2\sin(\alpha\cdot\pi)\cos\theta_1, \sin\theta_1\cos\theta_2-\sin\theta_2\cos(\alpha\cdot\pi)\cos\theta_1, -\sin\theta_1\sin\theta_2\sin(\alpha\cdot\pi))^{\mathrm{T}}$, we know that \eqref{thm2-2-1} can be written more precisely as 
	\begin{align}\label{thm2-2-2}
	&\frac{\partial u}{\partial\nu}+\eta_3u\Big|_{\vec{a}}
	=\frac{1}{r}\frac{\partial u}{\partial \theta}\sin\theta_2\sin(\alpha\cdot\pi)+\frac{1}{r\sin\theta_1}\frac{\partial u}{\partial \phi}\notag\\
	&\cdot(\sin\theta_1\cos\theta_2-\sin\theta_2\cos\theta_1\cos\alpha\cdot\pi)+\eta_3u\Big|_{\theta=\theta_1, \phi=0}=0.
	\end{align}
	By Lemma \ref{expansion}, multiplying $r\sin\theta_1$ on the both sides of \eqref{thm2-2-2}, the equation can be simplified as
	\begin{align}\label{thm2-2-3}
	&\sin\theta_1\sin\theta_2\sin(\alpha\cdot\pi)\sum_{n=0}^{\infty}\sum_{m=-n}^{n}\bsi^na_n^mj_n(\sqrt{\lambda}r)\sqrt{\frac{2n+1}{4\pi}}\sqrt{\frac{(n-|m|)!}{(n+|m|)!}}\frac{dP_n^{|m|}(\cos\theta)}{d\theta}\Big|_{\theta=\theta_1}\notag\\
	&+(\sin\theta_1\cos\theta_2-\sin\theta_2\cos\theta_1\cos\alpha\cdot\pi)\sum_{n=0}^{\infty}\sum_{m=-n}^{n}\bsi^{n+1}ma_n^mj_n(\sqrt{\lambda}r)\sqrt{\frac{2n+1}{4\pi}}\sqrt{\frac{(n-|m|)!}{(n+|m|)!}}\notag\\
	&\cdot P_n^{|m|}(\cos\theta_1)+\eta_3\sin\theta_1r\sum_{n=0}^{\infty}\sum_{m=-n}^n\bsi^na_n^mj_n(\sqrt{\lambda}r)\sqrt{\frac{2n+1}{4\pi}}\sqrt{\frac{(n-|m|)!}{(n+|m|)!}}P_n^{|m|}(\cos\theta_1)=0.
	\end{align}
    Since $u(\mathbf{ 0})=0$, we know that $a_0^0=0$. Combining \eqref{thm1-2-1}, \eqref{thm1-2-2} with \eqref{thm2-2-3}, the corresponding coefficients of $r$ respectively fulfil that 
	\begin{equation}\label{thm2-r1}
    \sum_{m=-1}^{1}ma_1^m\frac{\sqrt{\lambda}}{3!!}\sqrt{\frac{3}{4\pi}}\sqrt{\frac{(1-|m|)!}{(1+|m|)!}}P_1^{|m|}(\cos\theta)+\eta_1\sin\theta a_0^0\sqrt{\frac{1}{4\pi}}P_0^0(\cos\theta)=0,
    \end{equation}
    \begin{equation}\label{thm2-r12}
    \sum_{m=-1}^{1}ma_1^m\frac{\sqrt{\lambda}}{3!!}\sqrt{\frac{3}{4\pi}}\sqrt{\frac{(1-|m|)!}{(1+|m|)!}}P_1^{|m|}(\cos\theta)e^{\bsi m\alpha\cdot\pi}
    -\eta_2\sin\theta\sum_{m=-1}^{1}a_0^0\sqrt{\frac{1}{4\pi}}P_0^0(\cos\theta)e^{\bsi m\alpha\cdot\pi}=0,
    \end{equation}
    and
    \begin{align}\label{thm2-r13}
    &\sin\theta_1\sin\theta_2\sin(\alpha\cdot\pi)\sum_{m=-1}^{1}\bsi a_1^m \frac{\sqrt{\lambda}}{3!!}\sqrt{\frac{3}{4\pi}}\sqrt{\frac{(1-|m|)!}{(1+|m|)!}}\frac{dP_1^{|m|}(\cos\theta)}{d\theta}\Big|_{\theta=\theta_1}\notag\\
    &-(\sin\theta_1\cos\theta_2-\sin\theta_2\cos\theta_1\cos(\alpha\cdot\pi))\sum_{m=-1}^{1}ma_1^m\frac{\sqrt{\lambda}}{3!!}\sqrt{\frac{3}{4\pi}}\sqrt{\frac{(1-|m|)!}{(1+|m|)!}}P_1^{|m|}(\cos\theta_1)\notag\\
    &+\eta_3\sin\theta_1a_0^0\sqrt{\frac{1}{4\pi}}P_0^{0}(\cos\theta_1)=0.
    \end{align}
    Substituting $a_0^0=0$ into \eqref{thm2-r1} and \eqref{thm2-r12}, utilizing the orthogonality condition we can derive that
    \begin{equation}\label{thm2-a1}
    a_1^1-a_1^{-1}=0,\ \
    a_1^1e^{\bsi\alpha\cdot\pi}-a_1^{-1}e^{-\bsi\alpha\cdot\pi}=0,
    \end{equation}
    which yields $a_1^{\pm1}=0$ from the fact that $\alpha\in (0, 1)$. In addition, in \eqref{thm2-r13}, taking $a_0^0=a_1^{\pm1}=0$, we have
    \begin{equation}\label{thm2-r132}
    \sin\theta_1\sin\theta_2\sin(\alpha\cdot\pi)\bsi a_1^0\frac{\sqrt{\lambda}}{3!!}\sqrt{\frac{3}{4\pi}}(-P_1^1(\cos\theta_1))=0.
    \end{equation}
    Hence, by the assumptions on $\theta_1, \theta_2$ and $\alpha$, we can obtain that $a_1^0=0$ if $P_1^1(\cos\theta_1)\neq0$.
    
    Proving by induction, we assume that $a_{n-1}^m=0$ for $m=0,\pm1,\pm2,\cdots\pm(n-1)$. Then considering the coefficients of $r^n$ in \eqref{thm1-2-1}, \eqref{thm1-2-2} and \eqref{thm2-2-3} accordingly we know that
    \begin{align}\label{thm2-rn1}
    &-\sum_{m=-n}^{n}\bsi^{n+1}ma_n^m\frac{\sqrt{\lambda}^n}{(2n+1)!!}\sqrt{\frac{2n+1}{4\pi}}\sqrt{\frac{(n-|m|)!}{(n+|m|)!}}P_n^{|m|}(\cos\theta)\notag\\
    &+\eta_1\sin\theta\sum_{m=-(n-1)}^{n-1}\bsi^{n-1}a_{n-1}^m\frac{\sqrt{\lambda}^{n-1}}{(2n-1)!!}\sqrt{\frac{2n-1}{4\pi}}\sqrt{\frac{(n-1-|m|)!}{(n-1+|m|)!}}P_{n-1}^{|m|}(\cos\theta)=0,
    \end{align}
    \begin{align}\label{thm2-rn2}
    &\sum_{m=-n}^{n}\bsi^{n+1}ma_n^m\frac{\sqrt{\lambda}^n}{(2n+1)!!}\sqrt{\frac{2n+1}{4\pi}}\sqrt{\frac{(n-|m|)!}{(n+|m|)!}}P_n^{|m|}(\cos\theta)e^{\bsi m\alpha\cdot\pi}\notag\\
    &+\eta_2\sin\theta\sum_{m=-(n-1)}^{n-1}\bsi^{n-1}a_{n-1}^m\frac{\sqrt{\lambda}^{n-1}}{(2n-1)!!}\sqrt{\frac{2n-1}{4\pi}}\sqrt{\frac{(n-1-|m|)!}{(n-1+|m|)!}}P_{n-1}^{|m|}(\cos\theta)e^{\bsi m \alpha\cdot\pi}=0,
    \end{align}
    and
    \begin{align}\label{thm2-rn3}
    &\sin\theta_1\sin\theta_2\sin(\alpha\cdot\pi)\sum_{m=-n}^{n}\bsi^na_n^m\frac{\sqrt{\lambda}^n}{(2n+1)!!}\sqrt{\frac{2n+1}{4\pi}}\sqrt{\frac{(n-|m|)!}{(n+|m|)!}}\frac{dP_n^{|m|}(\cos\theta)}{d\theta}\Big|_{\theta=\theta_1}\notag\\
    &+(\sin\theta_1\cos\theta_2-\sin\theta_2\cos\theta_1\cos\alpha\cdot\pi)\sum_{m=-n}^{n}\bsi^{n+1}ma_n^m\frac{\sqrt{\lambda}^n}{(2n+1)!!}\sqrt{\frac{2n+1}{4\pi}}\sqrt{\frac{(n-|m|)!}{(n+|m|)!}} \notag\\
    &\cdot P_n^{|m|}(\cos\theta_1)+\eta_3\sin\theta_1\sum_{m=-(n-1)}^{n-1}\bsi^{n-1}a_{n-1}^m\frac{\sqrt{\lambda}^{n-1}}{(2n-1)!!}\sqrt{\frac{2n-1}{4\pi}}\sqrt{\frac{(n-1-|m|)!}{(n-1+|m|)!}}P_{n-1}^{|m|}(\cos\theta_1)=0.
    \end{align}
    Using the assumption that $a_{n-1}^m=0$, $m=0,\pm1,\pm2,\cdots\pm(n-1)$ in \eqref{thm2-rn1} and \eqref{thm2-rn2}, similar to Theorem \ref{thm1-1}, we can obtain that if $\alpha\neq\frac{k}{m}$, $k=1,2,\cdots,m-1$, then $a_n^m=0$ for $m=\pm1,\pm2,\cdots\pm n$. Therefore, from \eqref{thm2-rn3}, we can deduce that
    \begin{equation}\label{thm2-an0}
    \sin\theta_1\sin\theta_2\sin(\alpha\cdot\pi)\bsi^na_n^0\frac{\sqrt{\lambda}^n}{(2n+1)!!}\sqrt{\frac{2n+1}{4\pi}}(-P_n^1(\cos\theta_1))=0,
    \end{equation}
    which indicates that if $P_n^1(\cos\theta_1)\neq0$, then $a_n^0=0$.
     
    The proof is complete.
\end{proof}

\begin{remark}
	Following a similar argument in Theorem \ref{3d vertex-thm2}, if we take into account the condition $\frac{\partial u}{\partial\nu}+\eta_3u\Big|_{\vec{b}}\equiv0$ on $\Pi_3$, then we can derive the same results with respect to $\theta_2$ instead of $\theta_1$.
\end{remark}

\begin{remark}\label{rem:ll1}
By direct verifications in the proof of Theorem~\ref{3d vertex-thm2}, one can show that either one of the boundary parameters $\eta_\ell$, $\ell=1,2,3$, can be taken to be zero. That means, the generalised singular planes in Theorem~\ref{3d vertex-thm1} can be replaced to be singular planes, and the vanishing results still hold. 
\end{remark}

In Theorems \ref{3d vertex-thm1} and \ref{3d vertex-thm2}, we consider the vanishing properties at a vertex-corner point that is intersected by three planes $(n=3)$. 
In fact, the similar arguments work for the case that $n>3$, in which the third plane no longer intersects with $\Pi_1$ or $\Pi_2$. Without loss of generality, we denote the third plane to be discussed by $\Pi_j=\mbox{span}\{\overrightarrow{OA_j}, \overrightarrow{OA_{j+1}}\}$, where $3\leq j \leq n$ and if $j=n$, we assume that $A_{n+1}:=A_1$. Let $\Pi_1$ coincide with the $(x_1, x_2)$-plane, $\Pi_2$ possesses a dihedral angle $\alpha\cdot\pi$ away from $\Pi_1$ in the anti-clockwise direction and $\overrightarrow{OA_2}$ lies on the $x_3$-axis; see Figure \ref{fig:3d-vertex3} for a schematic illustration.

\begin{figure}[htbp]
	\centering
	\includegraphics[width=0.3\linewidth]{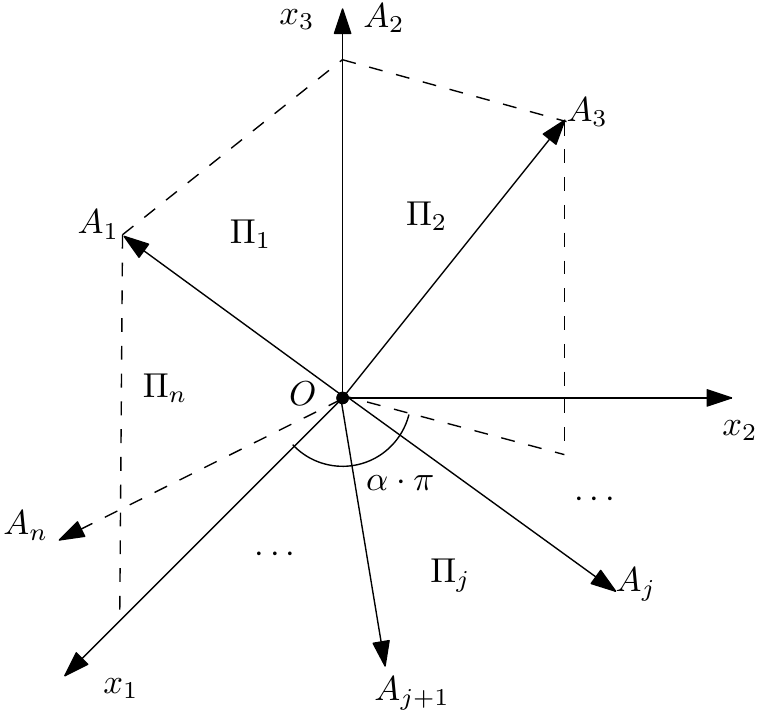}
	\caption{Schematic illustration of a vertex corner that is intersected by $\Pi_1$, $\Pi_2$, $\cdots$, $\Pi_n$ with $n>3$.}
	\label{fig:3d-vertex3}
\end{figure}

\begin{theorem}\label{3d vertex-n plane}
Let $u$ be a Laplacian eigenfunction to \eqref{eq:eig}. Consider a vertex corner ${\mathcal V}(\{\Pi_\ell\}_{\ell=1}^n,{\mathbf 0})\Subset \Omega$ as described above with $ \Pi_\ell  \in \mathcal{M}^\lambda_\Omega $, $\ell=1,2$, $\angle(\Pi_{1},\Pi_2)=\phi=\alpha\cdot\pi$, $\alpha\in(0,1)$, and $\Pi_j\in\mathcal{N}_\Omega^\lambda$. Assume that $\Pi_j=\mathrm{span}\{\overrightarrow{OA_j}, \overrightarrow{OA_{j+1}}\}$, where $\overrightarrow{OA_j}=(r, \theta_j, \phi_j)$ and $\overrightarrow{OA_{j+1}}=(r, \theta_{j+1}, \phi_{j+1})$ for $r>0$, $\theta_j,\theta_{j+1}\in(0, \pi)$, and $\phi_j, \phi_{j+1}\in(0, 2\pi)$ such that $0<\phi_{j+1}-\phi_j<\pi$ in the spherical coordinate system. If for an $N\in\mathbb{N}$, $N\geq 3$, there holds 
\begin{equation}\notag
  P_p^0(\cos\theta_\tau)\neq0, \tau=j\mbox{ or }j+1,\mbox{ and }\alpha\neq \frac{q}{p}, p=1,2,\cdots,N-1,\ q=1, 2,\cdots, p-1,
\end{equation}
where $P_p^0$ is the Legendre polynomial, then the vanishing order of $u$ at $\mathbf 0$ generated by the intersection of the two planes $\Pi_1$ and $\Pi_2$ is at least up to order $N$. 
\end{theorem}

\begin{proof}
	Since $\Pi_1$ and $\Pi_2$ are two generalized singular planes , we can derive \eqref{thm1-2-1} and \eqref{thm1-2-2} immediately. Considering $\Pi_j$, we know that $u|_{\Pi_j}=0$, which indicates that $u|_{\overrightarrow{OA_j}}\equiv0$ and $u|_{\overrightarrow{OA_{j+1}}}\equiv0$. By Remark \ref{vec a}, it suffices to analyze $u|_{\overrightarrow{OA_j}}\equiv0$ as follows:
	\begin{equation}\label{thm4-1}
	u|_{\overrightarrow{OA_j}}=4\pi\sum_{n=0}^{\infty}\sum_{m=-n}^{n}\bsi^na_n^mj_n(\sqrt{\lambda}r)\sqrt{\frac{2n+1}{4\pi}}\sqrt{\frac{(n-|m|)!}{(n+|m|)!}}P_n^{|m|}(\cos\theta_j)e^{\bsi m\phi_j}=0.
	\end{equation}
	Taking $m=n=0$ in \eqref{thm4-1} we have $4\pi a_0^0\sqrt{\frac{1}{4\pi}}P_0^0(\cos\theta_j)=0$,
	where we can derive $a_0^0=0$ since $P_0^0\equiv1$. Thus from \eqref{thm1-2-1}, \eqref{thm1-2-2} and \eqref{thm4-1}, we know that the coefficient of $r$ satisfies \eqref{thm1-a1} and thus $a_1^{\pm1}=0$. Moreover, we have
	\begin{equation}\label{thm4-a1}
	\sum_{m=-1}^1\bsi a_1^m\frac{\sqrt{\lambda}}{3!!}\sqrt{\frac{3}{4\pi}}\sqrt{\frac{(1-|m|)!}{(1+|m|)!}}P_1^{|m|}(\cos\theta_j)e^{\bsi m \phi_j}=0,
	\end{equation}
	which can be further simplified as $a_1^0P_1^0(\cos\theta_j)=0$
	after substituting $a_1^{\pm1}=0$ into \eqref{thm4-a1}. Hence, it is easy to see that $a_1^0=0$ if $P_1^0(\cos\theta_j)\neq0$.
	
	By induction, we assume that $a_{n-1}^m=0$ for $m=0,\pm1,\pm2,\cdots\pm(n-1)$. Considering the coefficient of $r^n$, we can obtain \eqref{thm1-rn1} and \eqref{thm1-rn2} which induce \eqref{thm1-an} as well as the following equation
	\begin{equation}\label{thm4-an}
	\sum_{m=-n}^n\bsi^na_n^m\frac{\sqrt{\lambda}^n}{(2n+1)!!}\sqrt{\frac{2n+1}{4\pi}}\sqrt{\frac{(n-|m|)!}{(n+|m|)!}}P_n^{|m|}(\cos\theta_j)e^{\bsi m\phi_j}=0.
	\end{equation}
	Since we have already known that if $\alpha\neq\frac{k}{m}$, $k=1,2,\cdots-m-1$, then $a_n^m=0$ for $m=\pm1,\pm2,\cdots,\pm n$ from \eqref{thm1-an}. Substituting this result into \eqref{thm4-an}, we can deduce that 
	\begin{equation*}\label{thm4-an2}
	a_n^0P_n^0(\cos\theta_j)=0.
	\end{equation*}
	Therefore, we can derive that $a_n^0=0$ if $\theta_j$ fulfills that $P_n^0(\cos\theta_j)\neq0$. Similarly, if we utilize the condition $u|_{\overrightarrow{OA_{j+1}}}\equiv0$, then the same argument and results work for $\theta_{j+1}$, which completes the proof.
\end{proof}

We proceed to consider the case that $\Pi_j$ is a generalised singular plane instead of being nodal in Theorem~\ref{3d vertex-n plane}. We have


\begin{theorem}\label{3d vertex-n plane22}
Let $u$ be a Laplacian eigenfunction to \eqref{eq:eig}. Consider a vertex corner ${\mathcal V}(\{\Pi_\ell\}_{\ell=1}^n,{\mathbf 0})\Subset \Omega$ as described before with $ \Pi_\ell  \in \mathcal{M}^\lambda_\Omega $, $\ell=1,2$, $\angle(\Pi_{1},\Pi_2)=\phi=\alpha\cdot\pi$, $\alpha\in(0,1)$, and $\Pi_j\in\mathcal{M}_\Omega^\lambda$. Assume that $\Pi_j=\mathrm{span}\{\overrightarrow{OA_j}, \overrightarrow{OA_{j+1}}\}$, where $\overrightarrow{OA_j}=(r, \theta_j, \phi_j)$ and $\overrightarrow{OA_{j+1}}=(r, \theta_{j+1}, \phi_{j+1})$ for $r>0$, $\theta_j,\theta_{j+1}\in(0, \pi)$, and $\phi_j, \phi_{j+1}\in(0, 2\pi)$ such that $0<\phi_{j+1}-\phi_j<\pi$ in the spherical coordinate system. If for an $N\in\mathbb{N}$, $N\geq 3$, there holds 
\begin{equation}\notag
  u(\mathbf{0})=0,\ \ P_p^0(\cos\theta_\tau)\neq0, \tau=j\mbox{ or }j+1,\mbox{ and }\alpha\neq \frac{q}{p},
\end{equation}
where $ p=1,2,\cdots,N-1,\ q=1, 2,\cdots, p-1$ and $P_p^1$ is the Legendre polynomial, then the vanishing order of $u$ at $\mathbf 0$ generated by the intersection of the two planes $\Pi_1$ and $\Pi_2$ is at least up to order $N$. 
\end{theorem}

\begin{proof}
	From Theorem \ref{3d vertex-thm2}, we know that since $\Pi_1$ and $\Pi_2$ are two generalized singular planes, then $u$ fulfils \eqref{thm1-2-1} and \eqref{thm1-2-2} accordingly. Now consider $\Pi_j$, there holds $\frac{\partial u}{\partial \nu}+\eta_ju=0$ on $\Pi_j$. Since $\Pi_j=\mbox{span}\{\overrightarrow{OA_j}, \overrightarrow{OA_{j+1}}\}$, we have $\frac{\partial u}{\partial\nu}+\eta_ju\Big|_{\overrightarrow{OA_j}}=0$ and
	\begin{equation}\notag
	\nu=\overrightarrow{OA_j}\times\overrightarrow{OA_{j+1}}=
	\left(
	\begin{array}{c}
	\sin\theta_j\sin\phi_j\cos\theta_{j+1}-\sin\theta_{j+1}\sin\phi_{j+1}\cos\theta_j\\
	-\sin\theta_j\cos\phi_j\cos\theta_{j+1}+\sin\theta_{j+1}\cos\phi_{j+1}\cos\theta_j\\
	\sin\theta_j\cos\phi_j\sin\theta_{j+1}\sin\phi_{j+1}-\sin\theta_{j+1}\cos\phi_{j+1}\sin\theta_j\sin\phi_j
	\end{array}
	\right).
	\end{equation} 
	Combining with Lemma \ref{expansion}, by direct computations, we can obtain
	\begin{align}\label{thm5-1}
	&\frac{\partial u}{\partial\nu}+\eta_ju\Big|_{\overrightarrow{OA_j}}=\frac{1}{r\sin\theta_j}\frac{\partial u}{\partial\phi}(\sin\theta_{j+1}\cos\theta_j\cos(\phi_j-\phi_{j+1})-\sin\theta_j\cos\theta_{j+1})\notag\\
	&+\frac{1}{r}\frac{\partial u}{\partial \theta}\sin\theta_{j+1}\sin(\phi_j-\phi_{j+1})+\eta_j u\Big|_{\theta=\theta_j, \phi=\phi_j}=0.
	\end{align}
	Since $\theta_j\in(0, \pi)$, multiplying $r\sin\theta_j$ on the both sides of \eqref{thm5-1}, we can deduce that 
	\begin{align}\label{thm5-2}
	&(\sin\theta_{j+1}\cos\theta_j\cos(\phi_j-\phi_{j+1})-\sin\theta_j\cos\theta_{j+1})\sum_{n=0}^{\infty}\sum_{m=-n}^n\bsi^{n+1}ma_n^mj_n(\sqrt{\lambda}r)\sqrt{\frac{2n+1}{4\pi}}\notag\\
	&\cdot\sqrt{\frac{(n-|m|)!}{(n+|m|)!}}P_n^{|m|}(\cos\theta_j)e^{\bsi m \phi_j}+\sin\theta_j\sin\theta_{j+1}\sin(\phi_j-\phi_{j+1})\sum_{n=0}^{\infty}\sum_{m=-n}^n\bsi^na_n^mj_n(\sqrt{\lambda}r)\notag\\
	&\cdot\sqrt{\frac{2n+1}{4\pi}}\sqrt{\frac{(n-|m|)!}{(n+|m|)!}}\frac{dP_n^{|m|}(\cos\theta)}{d\theta}\Big|_{\theta=\theta_j}e^{\bsi m\phi_j}+\eta_j\sin\theta_jr\sum_{n=0}^{\infty}\sum_{m=-n}^n\bsi^na_n^mj_n(\sqrt{\lambda}r)\notag\\
	&\cdot\sqrt{\frac{2n+1}{4\pi}}\sqrt{\frac{(n-|m|)!}{(n+|m|)!}}P_n^{|m|}(\cos\theta_j)e^{\bsi m\phi_j}=0.
	\end{align}
	Since $u(\mathbf{ 0})=0$, we have $a_0^0=0$. Considering the coefficients with respect to $r$ in \eqref{thm1-2-1}, \eqref{thm1-2-2} and \eqref{thm5-2}, we know that $a_1^{\pm1}$ fulfills \eqref{thm1-a1} which induces that $a_1^{\pm1}=0$ since $\alpha\in(0, 1)$. Moreover, it is easy to see from \eqref{thm5-2} that
	\begin{equation*}\label{thm5-a1}
	\sin\theta_j\sin\theta_{j+1}\sin(\phi_j-\phi_{j+1})\bsi a_1^0\frac{\sqrt{\lambda}}{3!!}\sqrt{\frac{3}{4\pi}}(-P_1^1(\cos\theta_j))=0.
	\end{equation*} 
	Since $\theta_j, \theta_{j+1}\in(0, \pi)$ and $0<\phi_j-\phi_{j+1}<\pi$, we know that if $P_1^1(\cos\theta_j)\neq0$ then $a_1^0=0$.
	
	Similarly, we assume that $a_{n-1}^m=0$, $m=0,\pm1,\pm2,\cdots,\pm(n-1)$. Then combining with Theorem \ref{3d vertex-thm2}, we know that $a_n^m$ satisfies \eqref{thm2-rn1}, \eqref{thm2-rn2} and 
	\begin{align}\label{thm5-rn}
	&(\sin\theta_{j+1}\cos\theta_j\cos(\phi_j-\phi_{j+1})-\sin\theta_j\cos\theta_{j+1})\sum_{m=-n}^n\bsi^{n+1}ma_n^m\frac{\sqrt{\lambda}^n}{(2n+1)!!}\sqrt{\frac{2n+1}{4\pi}}\notag\\
	&\cdot\sqrt{\frac{(n+|m|)!}{(n-|m|)!}}P_n^{|m|}(\cos\theta_j)e^{\bsi m \phi_j}+\sin\theta_j\sin\theta_{j+1}\sin(\phi_j-\phi_{j+1})\sum_{m=-n}^n\bsi^na_n^m\frac{\sqrt{\lambda}^n}{(2n+1)!!}\notag\\
	&\cdot\sqrt{\frac{2n+1}{4\pi}}\sqrt{\frac{(n-|m|)!}{(n+|m|)!}}\frac{dP_n^{|m|}(\cos\theta)}{d\theta}\Big|_{\theta=\theta_j}e^{\bsi m\phi_j}+\eta_j\sin\theta_j\sum_{m=-(n-1)}^{n-1}\bsi^{n-1}a_{n-1}^m\frac{\sqrt{\lambda}^{n-1}}{(2n-1)!!}\notag\\
	&\cdot\sqrt{\frac{2n-1}{4\pi}}\sqrt{\frac{(n-1-|m|)!}{(n-1+|m|)!}}P_{n-1}^{|m|}(\cos\theta_j)e^{\bsi m\phi_j}=0.
	\end{align}
	
	In \eqref{thm2-rn1} and \eqref{thm2-rn2}, utilizing the assumption $a_{n-1}^m=0$ for $m=0,\pm1,\pm2,\cdots\pm(n-1)$, we know that if $\alpha\neq\frac{k}{m}$, $k=1,2,\cdots,m$, then $a_n^m=0$, $\pm1,\pm2,\cdots,\pm n$. Hence \eqref{thm5-rn} can be simplified as 
	\begin{equation*}\label{thm5-an}
	\sin\theta_j\sin\theta_{j+1}\sin(\phi_j-\phi_{j+1})\bsi^n a_n^0\frac{\sqrt{\lambda}^n}{(2n+1)!!}\sqrt{\frac{2n+1}{4\pi}}(-P_n^1(\cos\theta_j))=0.
	\end{equation*}
	Since $\theta_j, \theta_{j+1}\in(0, \pi)$ and $0<\phi_j-\phi_{j+1}<\pi$, we can derive that if $P_n^1(\cos\theta_j\neq0)$, then $a_n^0=0$. The same results work for $\theta_{j+1}$ if we take into account $\frac{\partial u}{\partial \nu}+\eta_ju\Big|_{\overrightarrow{OA_{j+1}}}=0$.
	
	The proof is complete.
\end{proof}

\begin{remark}\label{rem:ll2}
Similar to Remark~\ref{rem:ll1}, one can have by direct verifications that the vanishing results in Theorem~\label{3d vertex-n plane2} still hold if any of the generalised singular planes involved is replaced to be a singular plane. We shall not present those results in order to avoid repetition. 
\end{remark}

\section{Irrational intersections and infinite vanishing orders}\label{sec4}

%

From the results derived in Sections \ref{sec2} and \ref{sec3}, one can identify that the vanishing order of the eigenfunction $u$ at an edge or a vertex corner point relies on the degree of the dihedral angle of the underlying corner. In the following two definitions, we first introduce the irrational and rational edge or vertex corner. Then, based on the results in Sections \ref{sec2} and \ref{sec3}, we show that the vanishing order of the eigenfunction at an irrational edge or vertex corner point is generically infinity and hence it is identically vanishing in $\Omega$.

\begin{definition}
	{\color{black}{
	Let ${\mathcal E}(\Pi_1, \Pi_2,\bsl)$ be an edge corner defined in Definition \ref{def:edge} and the corresponding dihedral angle of $\Pi_1$ and $\Pi_2$ is denoted by $\phi=\alpha \cdot \pi $, $\alpha\in(0,1)$. }}
	If $\phi$ is an irrational dihedral angle, namely, $\alpha$ is an irrational number, then ${\mathcal E}(\Pi_1, \Pi_2,\bsl)$  is said to be an {\it irrational} edge corner. Otherwise, it is said to be  a {\it rational} edge corner. For a rational edge corner ${\mathcal E}(\Pi_1, \Pi_2,\bsl)$, the dihedral angle between $\Pi_1$ and $\Pi_2$ is called the \emph{rational degree} of ${\mathcal E}(\Pi_1, \Pi_2,\bsl)$. 
\end{definition}

\begin{definition}\label{def:irration}
	Let ${\mathcal V}(\{\Pi_\ell\}_{\ell=1}^n,\bfx_0)$ be a vertex corner defined in Definition \ref{def:vertex}, where $n\in \N$ and $n\geq 3$. It is clear that, ${\mathcal V}(\{\Pi_\ell\}_{\ell=1}^n,\bfx_0)$ is composed by the following $n$ edge corners 
	$$
{\mathcal E}_\ell:=	{\mathcal E}(\Pi_\ell, \Pi_{\ell+1},\bsl_\ell), \quad {\mathcal E}_n:={\mathcal E}(\Pi_n, \Pi_{1},\bsl_n), \quad \Pi_{n+1}:=\Pi_1, \quad \ell=1,2,\ldots, n-1, 
	$$
	where $\bsl_\ell$ is the line segment of $\Pi_\ell \cap \Pi_{\ell+1}$ and $\bsl_n$ is a line segment of  $\Pi_n \cap \Pi_{1}$, respectively.  Denote
	\begin{equation}\label{eq:ir index}
	\begin{split}
			I_{\sf IR}&=\{\ell \in \N~|~1\leq \ell\leq n,\quad {\mathcal E}_\ell \mbox{ is an irrational edge corner}\},\\
			I_{\sf R}&=\{\ell \in \N~|~1\leq \ell\leq n,\quad {\mathcal E}_\ell \mbox{ is a rational edge corner}\}.  
	\end{split}
	\end{equation}
	If $\#I_{\sf IR}\geq 1$, then ${\mathcal V}(\{\Pi_\ell\}_{\ell=1}^n,\bfx_0)$ is said to be an {\it irrational } vertex corner. If $\#I_{\sf IR}\equiv 0$, then ${\mathcal V}(\{\Pi_\ell\}_{\ell=1}^n,\bfx_0)$ is said to be a {\it rational } vertex corner. For a { rational } vertex corner ${\mathcal V}(\{\Pi_\ell\}_{\ell=1}^n,\bfx_0)$ composed by edge corners ${\mathcal E}_\ell:={\mathcal E}(\Pi_\ell, \Pi_{\ell+1},\bsl_\ell)$, the {\color{black}{largest}} degree of ${\mathcal E}_\ell\, (\ell=1,\ldots, n)$ is referred to as the \emph{rational degree} of ${\mathcal V}(\{\Pi_\ell\}_{\ell=1}^n,\bfx_0)$ . 
\end{definition}



When an irrational edge corner  ${\mathcal E}(\Pi_1, \Pi_2,\bsl) $ is intersected by two nodal planes of $u$, from Theorem \ref{soundsoft2}, we have

\begin{theorem}\label{ir-2nodal}
	Let $u$ be a Laplacian eigenfunction to \eqref{eq:eig}. Suppose that ${\mathcal E}(\Pi_1, \Pi_2,\bsl) \Subset \Omega$ is an irrational edge corner with  $\Pi_1, \Pi_2\in\mathcal{N}_\Omega^\lambda$. 
%
	Then there holds
	\begin{equation*}\label{result1}
	\mathrm{Vani}(u; {\mathbf 0},\Pi_1, \Pi_2)=+\infty,\quad {\mathbf 0}\in \bsl. 
	\end{equation*}  
\end{theorem}
\begin{remark}
	Apart from Theorem \ref{soundsoft2} which gives the detailed analysis for this theorem with spherical wave expansion argument, we can prove Theorem \ref{ir-2nodal}  by reflection principle which is very similar to \cite[Theorem 2.1]{CDL2}. The detailed proof is omitted.

\end{remark}

If the intersecting two planes of the irrational 3D edge corner are either one of the three types: nodal plane, singular plane or generalized singular plane, namely for the general case, we have the irrational intersection results as follows.

\begin{theorem}\label{ir-mix1}
	Let $u$ be a Laplacian eigenfunction to \eqref{eq:eig}. Suppose that ${\mathcal E}(\Pi_1, \Pi_2,\bsl) \Subset \Omega$ is an irrational 3D edge corner with $\Pi_1\in\mathcal{N}_\Omega^\lambda$ and $\Pi_2\in\mathcal{M}_\Omega ^\lambda$ .  Assume that $\eta\equiv C$ on $\Pi_2$ is a constant, then there holds 
    \begin{equation*}\label{result2}
   \mathrm{Vani}(u; {\mathbf 0},\Pi_1, \Pi_2)=+\infty,\quad {\mathbf 0}\in \bsl. 
    \end{equation*}
\end{theorem}

The same result can be derived for the case $\eta\equiv0$, which indicates that $\Pi_2$ is a singular plane. The detailed discussion can be found in Theorem \ref{gene-nodal}.

The next theorem is concerned with the intersection of two generalized singular planes, which is a direct corollary of Theorem \ref{2generalize}. 

\begin{theorem}\label{ir-2gene}
	Let $u$ be a Laplacian eigenfunction to \eqref{eq:eig}. Suppose that ${\mathcal E}(\Pi_1, \Pi_2,\bsl) \Subset \Omega$ is an irrational edge corner with $\Pi_\ell\in\mathcal{M}_\Omega ^\lambda\, (\ell=1,2)$ . If there exits a sufficiently small $\varepsilon\in\mathbb{R}_+$ such that
	\begin{equation}\label{lem:27 cond}
		u|_{ B_\varepsilon(\mathbf{0}) \cap \bsl}\equiv0,
	\end{equation}
	 then there holds
	\begin{equation*}\label{result3}
	\mathrm{Vani}(u; {\mathbf 0},\Pi_1, \Pi_2)=+\infty,\quad {\mathbf 0}\in \bsl. 
	\end{equation*}
\end{theorem}


If $\eta_1=0$ or $\eta_2=0$, which indicates that either $\Pi_1$ or $\Pi_2$ becomes a singular plane, we can deduce the same vanishing property as Theorem \ref{ir-2gene}. Moreover, if $\eta_1=\eta_2=0$, for the intersection of two singular planes, we can further obtain the explicit form of $u$ as follow.
 
\begin{theorem}\label{ir-2sing}
	Let $u$ be a Laplacian eigenfunction to \eqref{eq:eig}.  Suppose that ${\mathcal E}(\Pi_1, \Pi_2,\bsl) \Subset \Omega$ is an irrational edge corner and  $\Pi_\ell\in\mathcal{S}_\Omega ^\lambda\, (\ell=1,2)$ . If \eqref{lem:27 cond} is satisfied, then there holds
	\begin{equation}\label{result4}
	\mathrm{Vani}(u; {\mathbf 0}, \Pi_1, \Pi_2)=+\infty,\quad  {\mathbf 0} \in \bsl . 	\end{equation}
	Moreover, if \eqref{lem:27 cond} fails to be fulfilled, then we have the following expansion of $u$ in a neighborhood of the edge-corner point $\mathbf{ 0}$ in the polar coordinate system:
	\begin{equation}\label{result5}
	u(\mathbf{ x})=4\pi\sum_{n=0}^{\infty}\bsi^na_n^0j_n(\sqrt{\lambda}r)Y_n^0(\theta, \phi), \quad 
	\end{equation}
	where $Y_n^0(\theta, \phi)$ is the spherical harmonics and $j_n(t)$ is the $n$-th Bessel function.
\end{theorem}

{\color{black}{
\begin{proof}
	By Theorem~\ref{2generalize} and Remark~\ref{rem:gg1}, it is easy to verify that \eqref{result4} holds under the generic condition \eqref{lem:27 cond}.
	However, if \eqref{lem:27 cond} fails to be fulfilled, then we can not derive $a_n^0=0$ for $n=0,1,2,\cdots$, therein. 
	
	Since $\frac{\partial u}{\partial \nu}\Big|_{\Pi_\ell}\equiv0$, $\ell=1,2$, by direct computation, we can obtain
	\begin{equation}\label{sing1}
	-\sum_{n=0}^{\infty}\sum_{m=-n}^{n}\bsi^{n+1}ma_n^mj_n(\sqrt{\lambda}r)\sqrt{\frac{2n+1}{4\pi}}\sqrt{\frac{(n-|m|)!}{(n+|m|)!}}P_n^{|m|}(\cos\theta)=0,
	\end{equation}
	and 
	\begin{equation}\label{sing2}
	\sum_{n=0}^{\infty}\sum_{m=-n}^{n}\bsi^{n+1}ma_n^mj_n(\sqrt{\lambda}r)\sqrt{\frac{2n+1}{4\pi}}\sqrt{\frac{(n-|m|)!}{(n+|m|)!}}P_n^{|m|}(\cos\theta)e^{\bsi m\alpha\cdot\pi}=0,
	\end{equation}
	on $\Pi_{1}$ and $\Pi_2$ respectively. By comparing the coefficient of $r$ in \eqref{sing1} and \eqref{sing2}, with the help of the orthogonality condition, we can still obtain that $a_1^{\pm1}=0$ since $\alpha\in(0,1)$ for the dihedral angle $\phi=\alpha\cdot\pi$. By induction, following a same argument to the proof of Theorem \ref{2generalize}, we can deduce that
	\begin{equation}\label{impe1-11}
	-\sum_{m=-n}^{n}\bsi^{n+1}ma_n^m\frac{\sqrt{\lambda}^n}{(2n+1)!!}\sqrt{\frac{2n+1}{4\pi}}\sqrt{\frac{(n-|m|)!}{(n+|m|)!}}P_n^{|m|}(\cos\theta)=0,
	\end{equation}
	and 
	\begin{equation}\label{impe2-11}
	\sum_{m=-n}^{n}\bsi^{n+1}ma_n^m\frac{\sqrt{\lambda}^n}{(2n+1)!!}\sqrt{\frac{2n+1}{4\pi}}\sqrt{\frac{(n-|m|)!}{(n+|m|)!}}P_n^{|m|}(\cos\theta)e^{\bsi m\alpha\cdot\pi}=0.
	\end{equation}
	Therefore, by Lemma \ref{base2}, we know that there holds $a_n^m=0$ ($m=\pm1,\pm2,\cdots,\pm n$) since the corresponding dihedral angle is irrational. Hence, we are able to obtain the explicit expression \eqref{result5} around the edge corner point $\mathbf{ 0}$. 
\end{proof}
}}

Based on the irrational intersection at an edge corner by two planes, we next consider the corresponding properties at a vertex corner which is intersected by $n$ planes where $n\geq3$.

%

Using Theorem \ref{3d vertex-n plane}, Theorem \ref{3d vertex-n plane2} and Remark \ref{rem:ll2}, for an irrational vertex corner, we have 


\begin{theorem}\label{ir-vertex}
	Let $u$ be a Laplacian eigenfunction to \eqref{eq:eig}. Consider an irrational vertex corner ${\mathcal V}(\{\Pi_\ell\}_{\ell=1}^n,{\mathbf 0})\Subset \Omega$,
	{\color{black}{where the intersecting $n$ planes $\Pi_1, \Pi_2, \cdots$, $\Pi_n$}} could be either one of the three types: nodal plane, singular plane or generalized singular plane. Assume that for $i=1,2,\cdots,n$, $\Pi_i=\mathrm{span}\{\overrightarrow{OA_i},\overrightarrow{OA_{i+1}}\}$, where $\overrightarrow{OA_i}= (r,\theta_i, \phi_i)$, $\overrightarrow{OA_{i+1}}=(r,\theta_{i+1}, \phi_{i+1})$ for $r>0$, $\theta_i,\theta_{i+1}\in(0,\pi)$ and $\phi_i, \phi_{i+1}\in(0,2\pi)$ such that {\color{black}{$0<\phi_{i+1}-\phi_i<\pi$}} in the spherical coordinate system. Particularly when $i=n$, we denote $\Pi_{n+1}:=\Pi_1$. Recall that $I_{\sf R} $ and $I_{\sf IR}$ are defined in \eqref{eq:ir index}. If one of the following conditions is fulfilled
\begin{enumerate}
	\item there exists an index $\ell_0 \in I_{\sf IR}$ such that $\Pi_{\ell_0} \in {\mathcal N}_\Omega^\lambda$ or $\Pi_{\ell_0+1} \in {\mathcal N}_\Omega^\lambda$;
	
	\item 
	{\color{black}{
	for any $\ell \in I_{\sf IR}$, if $\Pi_\ell, \Pi_{\ell+1}\in \{{\mathcal S}_\Omega^\lambda\cup{\mathcal M}_\Omega^\lambda \}$, $u({\mathbf 0})=0$ and for a fixed $\ell_0 \in I_{\sf IR}$ there exits an index $j\in \{ 1,\ldots, n\}$ }}such that the corresponding  plane $\Pi_j=\mathrm{span}\{\overrightarrow{OA_j}, \overrightarrow{OA_{j+1}}\}$ satisfies  $P_p^0(\cos\theta_\tau)\neq 0$ and $ P_p^1(\cos\theta_\tau)\neq 0$, for $\tau=j, j+1$, $p=1,2,\cdots, n-1$, where $P_p^0$ and $P_p^1$ are the associated Legendre polynomials, $n\in\mathbb{N}$, $n\geq3$;
\end{enumerate}
then  there holds 
$$
\mathrm{Vani}(u;\mathbf{ 0})=+\infty.
$$ 	
\end{theorem}

\section{Unique identifiability for inverse obstacle problem}\label{sec5}

In this section, for practical use, we apply the vanishing properties of the eigenfunction $u$ at a vertex-corner point to study the unique identifiability for the inverse obstacle problem which is concerned with recovering the shape of some unknown objects by certain wave probing data. The inverse obstacle problem arises from many applications such as radar, sonar and geophysical exploration.

Let $\Omega\subset\mathbb{R}^3$ be a bounded Lipschitz domain such that $\mathbb{R}^3\backslash\bar{\Omega}$ is connected. Let $u^i$ be an incident field and in the subsequent discussions of this section, $u^i$ is assumed to be a plane wave of the form 
\begin{equation*}\label{ui}
u^i:=u^i(\mathbf{ x};k,\mathbf{d})=e^{\bsi k\mathbf{ x}\cdot\mathbf{d}},\quad x\in\mathbb{R}^3,
\end{equation*}
where $k\in\mathbb{R}_+$ signifies the wavenumber and $\mathbf{d}\in\mathbb{S}^2$ denotes the incident direction. Physically speaking, $u^i$ is the detecting wave field and $\Omega$ denotes an impenetrable obstacle which interrupts the propagation of the incident wave and generates the corresponding scattered wave field $u^s$. Define  $u:=u^i+u^s$ to be the total wave field, then the forward scattering problem of this process can be described by the following system, 
\begin{equation}\label{forward}
\begin{cases}
& \Delta u + k^2 u = 0\qquad\quad \mbox{in }\ \ \mathbb{R}^3\backslash\overline{\Omega},\medskip\\
& u =u^i+u^s\hspace*{1.56cm}\mbox{in }\ \ \mathbb{R}^3,\medskip\\
& \mathscr{B}(u)=0\hspace*{1.95cm}\mbox{on}\ \ \partial\Omega,\medskip\\
&\displaystyle{ \lim_{r\rightarrow\infty}r^{\frac{1}{2}}\left(\frac{\partial u^{s}}{\partial r}-\mathrm{i}ku^{s}\right) =\,0,}
\end{cases}
\end{equation}
where the last equation is the Sommerfeld radiation condition that holds uniformly in $\hat{\mathbf{ x}}:=\mathbf{ x}/|\mathbf{ x}|\in\mathbb{S}^2$. If $\mathscr{B}(u):=u$, the boundary condition is of Dirichlet type and $\Omega$ is said to be a sound-soft obstacle. If $\mathscr{B}(u):=\partial_\nu u$, the boundary condition is of Neumann type and $\Omega$ is said to be a sound-hard obstacle. If $\mathscr{B}(u):=\partial_\nu u+\eta u$, $\Omega$ becomes an impedance obstacle with Robin type boundary condition where $\nu$ denotes the exterior unit normal vector to $\partial\Omega$ and $\eta\in L^\infty(\partial\Omega)$ signifies the corresponding impedance boundary parameter. For unification of the notation, we represent all these three types of boundary conditions with
\begin{equation}\label{bound}
\mathscr{B}(u):=\partial_\nu u+\eta u=0 \quad\mbox{on } \partial\Omega,
\end{equation}
where $\eta=\infty$ and $\eta=0$ respectively stands for the Dirichlet and Neumann boundary condition.

The forward scattering problem \eqref{forward} has been studied in \cite{CK, Mclean} and there exists a unique solution $u\in H^1_{loc}(\mathbb{R}^3\backslash\overline{\Omega})$ fulfilling the following expansion:
\begin{equation}\label{eq:far}
u^s(\mathbf x,\mathbf d,k)
=\frac{e^{\mathrm{i}kr}}{r^{{1/2}}} u_{\infty}(\hat{\mathbf x};k,\mathbf{d})+\mathcal{O}\left(\frac{1}{r^{3/2}}\right)\quad\mbox{as }\,r\rightarrow\infty,
\end{equation}
where $u_\infty$ is known as the associated far-field pattern or the scattering amplitude. The asymptotic form \eqref{eq:far} holds uniformly with respect to all directions $\hat {\mathbf x}:=\mathbf  x/|\mathbf  x|\in \mathbb{S}^{2}$. 

The inverse obstacle scattering problem corresponding to \eqref{forward} is to recover $\Omega$ (and $\eta$ as well in the impedance case) by knowledge of the far-field pattern $u_\infty(\hat{\mathbf{ x}},\mathbf{d},k)$. By introducing an operator $\mathcal{F}$ which sends the obstacle to the corresponding far-field pattern, defined by the forward scattering system \eqref{forward}, the aforementioned inverse problem can be formulated as
\begin{equation}\label{inverse}
\mathcal{F}(\Omega, \eta)=u_\infty(\hat{\mathbf{ x}},\mathbf{d},k).
\end{equation} 
It can directly verified that the inverse problem \eqref{inverse} is nonlinear. The problem is also known as the Schiffer problem in the inverse scattering theory, which has a long and colorful history since 1960 by M. Schiffer's pioneering work \cite{LP}. It constitutes an open problem whether one can establish the one-to-one correspondence for \eqref{inverse} by a single far-field pattern, namely $\hat{\mathbf{x}}$ while $k$ and $\mathbf{d}$ are fixed. We refer to a recent survey paper \cite{CK18} by Colton and Kress for more account on the historical developments of this problem. 

Recent progress on the Schiffer problem is made on general polyhedral obstacles in $\mathbb{R}^n$, $n\geq 2$. Uniqueness and stability results by using a finite number of far-field patterns can be found in \cite{AR,CY,LPRX,LRX,Liu-Zou,Liu-Zou3}. Particularly, in \cite{Liu-Zou3}, the unique determination for impedance-type obstacles was studied for partial solution to this fundamental problem. In \cite{CDL2}, we have developed a completely new method that is applicable for sound-soft, sound-hard and also impedance type obstacles to provide a solution to the inverse obstacle problem in two-dimensional space. We also showed that in a rather general scenario one can determine the convex hull of an impedance obstacle as well as its boundary parameter by at most two far-field patterns by utilizing this new approach. In this section, we are concerned with the recovery of the obstacle and its surface impedance in $\mathbb{R}^3$. Similar to the two-dimensional case, the method developed here is also completely local. Next, we give some basic definitions for the inverse obstacle problem in $\mathbb{R}^3$.

\begin{definition}\label{ad obstacle}
	Let $\Omega\subset\mathbb{R}^3$ be an open polyhedron associated with the generalized impedance boundary condition \eqref{bound}. Then $\Omega$ is said to be an \emph{admissible polyhedral obstacle} if the following conditions are fulfilled:
	\begin{itemize}
		\item On the each face of $\partial\Omega$, the surface impedance $\eta$ is either a constant (possibly zero) or $\infty$.
		
		\item For any vertex of $\Omega$ which is intersected by $n$ planes: $\Pi_1, \Pi_2,\cdots,\Pi_n$, $n\geq3$, there exists a plane $\Pi_j:=\{\overrightarrow{OA_j}, \overrightarrow{OA_{j+1}}\}$, where $O$ denotes the vertex centering at the origin, $\overrightarrow{OA_\tau}= (r,\theta_\tau, \phi_\tau)$, for $r>0$, $\theta_\tau\in(0,\pi)$ and $\phi_\tau\in(0,2\pi)$ in spherical coordinate system such that $P_n^0(\cos\theta_\tau)\neq0$ and $P_n^1(\cos\theta_\tau)\neq0$, where $\tau=j,j+1$, $n\in\mathbb{N}$, $P_n^0$ and $P_n^1$ are the Legendre polynomials.
	\end{itemize}
\end{definition}

\begin{remark}
Indeed, by the fact that $P_n^0(1)\equiv1$, $P_n^1(1)\equiv0$ when $\theta=0$ for all $n\in\mathbb{N}$, and the continuity of the Legendre polynomials, it is easy to know that there exist $\delta_0>0$ such that for any $\epsilon>0$ and $\theta\in(0, \delta_0)$, there holds $P_n^0(\cos\theta)\in(1-\epsilon,1)$ and $P_n^1(\cos\theta)\in(0,\epsilon)$, which implies the existence of $\theta_\tau$ in Definition \ref{ad obstacle}. Therefore, the definition of the admissible polyhedral obstacle is well-defined. 
\end{remark}

Throughout this section, we signify an admissible polyhedral obstacle as ($\Omega, \eta$). Then we define the rational and irrational obstacle in $\mathbb{R}^3$ based on Definition \ref{def:irration} for the rational and irrational {\color{black}{vertex corner}} of $\Omega$.

{\color{black}{
\begin{definition}\label{ir obstacle}
	Let $(\Omega, \eta)$ be an admissible polygonal obstacle. If  there exits a vertex corner that is rational, then it is said to be a \emph{rational obstacle}. If all the vertex corners of $\Omega$ are irrational, then it is called an \emph{irrational obstacle}. The smallest degree of the rational corner of $\Omega$ is referred to as the \emph{rational degree} of $\Omega$. 
\end{definition}  }}


\begin{definition}\label{def6}
	$\Omega$ is said to be an admissible complex polyhedral obstacle if it consists of finitely many admissible polyhedral obstacles. That is,
	\begin{equation*}\label{eq:r2a}
	(\Omega, \eta)=\bigcup_{j=1}^l (\Omega_j, \eta_j),
	\end{equation*}
	where $l\in\mathbb{N}$ and each $(\Omega_j, \eta_j)$ is an admissible polyhedral obstacle. Here, we define
	\begin{equation*}\label{eq:r2b}
	\eta=\sum_{j=1}^l \eta_j\chi_{\partial\Omega_j}. 
	\end{equation*}
	Moreover, $\Omega$ is said to be irrational if all of its component polyhedral obstacles are irrational, otherwise it is said to be rational. For the latter case, the smallest degree among all the degrees of its rational components is defined to be the degree of the complex obstacle $\Omega$.  
\end{definition}

Next, we give the unique determination result for an admissible complex irrational polyhedral obstacle by at most two far-field patterns. 

\begin{theorem}\label{inverse1}
	Let $(\Omega, \eta)$ and $(\widetilde\Omega, \widetilde\eta)$ be two admissible complex irrational obstacles. Let $k\in\mathbb{R}_+$ be fixed and $\mathbf{d}_\ell$, $\ell=1, 2$ be two distinct incident directions from $\mathbb{S}^2$.  Let $\mathbf{G}$ denote the unbounded connected component of $\mathbb{R}^3\backslash\overline{(\Omega\cup\widetilde\Omega)}$. Let $u_\infty$ and $\widetilde{u}_\infty$ be, respectively, the far-field patterns associated with $(\Omega, \eta)$ and $(\widetilde\Omega, \widetilde\eta)$. If 
	\begin{equation}\label{eq:cond1}
	u_\infty(\hat {\mathbf x}, \mathbf{d}_\ell )=\widetilde u_\infty(\hat {\mathbf x}, \mathbf{d}_\ell), \ \ \hat{\mathbf x}\in\mathbb{S}^2, \ell=1, 2,
	\end{equation}
	then one has that
	$$
	\left(\partial \Omega \backslash \partial \overline{ \widetilde{\Omega }} \right  )\bigcup \left(\partial \widetilde{\Omega } \backslash \partial  \overline{ \Omega } \right )
	$$
	cannot possess a vertex-corner point on $\partial \mathbf{G}$. Moreover,
	\begin{equation}\label{eta}
	\eta=\widetilde{\eta}\quad\mbox{on}\quad \partial\Omega\cap\partial\widetilde{\Omega}.
	\end{equation}

\end{theorem}


\begin{proof}
	    We prove the theorem by contradiction. Assume that
		$\left(\partial \Omega \backslash \partial \overline{ \widetilde{\Omega }} \right  )\cup \left (\partial \widetilde{\Omega } \backslash\partial  \overline { \Omega  } \right ) $ has a vertex-corner point $\mathbf x_c$ on $\partial \mathbf{G}$.  Then, $\mathbf x_c$ is either located at $\Omega$ or  $\widetilde\Omega$. Without loss of generality, we assume that $\mathbf x_c$ is a 3D vertex corner of $\widetilde\Omega$, which also indicates that $\mathbf{x}_c$ lies outside $\Omega$. Let $h\in\mathbb{R}_+$ be sufficiently small such that $B_h(x_c)\Subset\mathbb{R}^2\backslash\overline \Omega $, then we can suppose that 
		\begin{equation*}\label{eq:aa2}
		B_h(\mathbf x_c)\cap \partial\widetilde\Omega=\Pi_i,\quad i=1,2,\cdots,n,\quad n\geq3,
		\end{equation*}
		where $\Pi_i$ are the $n$ planes lying on the $n$ faces of $\widetilde\Omega$ that intersect at $\mathbf x_c$. 
%
%
		
		Recall that $\mathbf{G}$ denote the unbounded connected component of $\mathbb{R}^3\backslash\overline{(\Omega\cup\widetilde\Omega)}$. By \eqref{eq:cond1} and the Rellich theorem (cf. \cite{CK}), we know that
		\begin{equation}\label{eq:aa3}
		u(\mathbf x; k, \mathbf{d}_\ell)=\widetilde{u}(\mathbf x; k, \mathbf{d}_\ell),\quad x\in\mathbf{G},\ \ell=1, 2. 
		\end{equation}
		Since $\Pi_i\subset\partial\mathbf{G}$, $i=1,2,\cdots,n$, combining \eqref{eq:aa3} with the generalized boundary condition \eqref{bound} on $\partial\widetilde\Omega$, it is easy to obtain that
		\begin{equation*}\label{eq:aa4}
		\partial_\nu u+\widetilde\eta u=\partial_\nu \widetilde u+\widetilde\eta\widetilde u=0\quad\mbox{on}\ \ \Pi_i,\quad i=1,2,\cdots,n,\quad n\geq3. 
		\end{equation*}
		Furthermore, since $B_h(x_c)\Subset\mathbb{R}^2\backslash\overline \Omega $, we have $-\Delta u=k^2 u$ in $B_h(\mathbf x_c)$. We next divide our proof into two separate cases. 
		
	
	\medskip
	
	\noindent {\bf Case 1.}~Suppose that either $u(\mathbf x_c; k, \mathbf{d}_1)$ or $u(\mathbf x_c; k, \mathbf{d}_2)$ is zero. Without loss of generality, we assume that $u(\mathbf x_c; k, \mathbf{d}_1)=0$. By the assumption of the theorem that $\widetilde\Omega$ is an admissible  irrational obstacle, we can know that $\mathbf{ x}_c$ is an irrational vertex-corner point of $\widetilde{\Omega}$, which also implies that 
	{\color{black}{there exists $\Pi_{i_0}$ and $\Pi_{i_0+1}$ such that the corresponding intersecting dihedral angle is irrational.}} Hence, by our results in Sections~\ref{sec3} and \ref{sec4}, we can immediately derive that
	\begin{equation*}\label{eq:aa5}
	u(\mathbf x; k, \mathbf{d}_1)=0\quad\mbox{in}\ \ B_h(\mathbf x_c),
	\end{equation*}
	which in turn yields by the analytic continuation that 
	\begin{equation}\label{eq:aa51}
	u(\mathbf x; k, \mathbf{d}_1)=0\quad\mbox{in}\ \ \mathbb{R}^2\backslash\overline{\Omega}. 
	\end{equation}
	In particular, one has from \eqref{eq:aa51} that
	\begin{equation}\label{eq:aa6}
	\lim_{|\mathbf x|\rightarrow\infty} \left|u(\mathbf x; k, \mathbf{d}_1)\right|=0. 
	\end{equation}
	But this contradicts to the fact that follows from \eqref{eq:far}:
	\begin{equation}\label{eq:aa61}
	\lim_{|\mathbf x|\rightarrow\infty} \left|u(\mathbf  x; k, \mathbf{d}_1)\right|=\lim_{|\mathbf x|\rightarrow\infty} \left|e^{\mathrm{i}k\mathbf x\cdot \mathbf{d}_1}+u^s(\mathbf x; k, \mathbf{d}_1)\right|=1. 
	\end{equation}
	
	\medskip
	
	\noindent {\bf Case 2.}~ Suppose that both $u(\mathbf x_c; k, \mathbf{d}_1)\neq 0$ and $u(\mathbf x_c; k, \mathbf{d}_2)\neq 0$. Set
	\begin{equation}\label{eq:bb1}
	\alpha_1=u(\mathbf x_c; k, \mathbf{d}_2)\quad\mbox{and}\quad \alpha_2=-u(\mathbf x_c; k, \mathbf{d}_1),
	\end{equation}
	and
	\begin{equation}\label{eq:bb2}
	v(\mathbf x)=\alpha_1 u(\mathbf x; k, \mathbf{d}_1)+\alpha_2 u(\mathbf x; k, \mathbf{d}_2),\ \ \ x\in B_h(\mathbf x_c). 
	\end{equation}
	It is easy to verify that $v$ fulfills
		\begin{equation}\label{eq:bb3}
		-\Delta v=k^2 v\quad\mbox{in}\ \ B_h(\mathbf x_c) \quad\mbox{and}\quad
		\partial_\nu v+\widetilde\eta v=0\quad\mbox{on}\ \ \Pi_i, i=1,2,\cdots,n, n\geq3. 
		\end{equation}
	Moreover, by the choice of $\alpha_1, \alpha_2$ in \eqref{eq:bb1}, one obviously has that $v(\mathbf x_c)=0$. Hence, by our results in Sections~\ref{sec3} and \ref{sec4}, we can deduce that
	\begin{equation*}\label{eq:bb5}
	v=0\quad\mbox{in}\ \ B_h(\mathbf x_c),
	\end{equation*}
	and thus
	\begin{equation}\label{eq:bb6a}
	\alpha_1 u(\mathbf x; k, \mathbf{d}_1)+\alpha_2 u(\mathbf x; k, \mathbf{d}_2)=0\quad\mbox{in}\ \ \mathbb{R}^3\backslash\overline{\Omega}
	\end{equation}
	by the analytic continuation.
	However, since $\mathbf{d}_1$ and $\mathbf{d}_2$ are distinct, we know from \cite[Chapter 5]{CK} that $u(x; k, \mathbf{d}_1)$ and $u(x; k, \mathbf{d}_2)$ are linearly independent in $\mathbb{R}^3\backslash\overline{\Omega}$.  Therefore, from \eqref{eq:bb6a} we can obtain that $\alpha_1=\alpha_2=0$, which contracts to the assumption at the beginning that both $\alpha_1$ and $\alpha_2$ are nonzero. 
	
	Then we prove \eqref{eta} by contradiction.
	Let $\mathcal{E}\subset \partial\Omega\cap\partial\widetilde\Omega$ be an open subset such that $\eta\neq \widetilde\eta$ on $\mathcal{E}$. By taking a smaller subset of $\mathcal{E}$ if necessary, we can assume that $\eta$ (respectively, $\widetilde\eta$) is either a fixed constant or $\infty$ on $\mathcal{E}$. Clearly, one has $u=\widetilde u$ in $\mathbb{R}^3\backslash\overline{(\Omega\cup\widetilde\Omega)}$. Hence, there holds that 
	\begin{equation*}\label{eq:bb6}
	\partial_\nu u+\eta u=0, \ \ \partial_\nu\widetilde u+\widetilde\eta \widetilde u=0,\ \ u=\widetilde u, \ \ \partial_\nu u=\partial_\nu\widetilde u\quad\mbox{on}\ \ \mathcal{E}. 
	\end{equation*}
	Combining with the assumption that $\eta\neq\widetilde{\eta}$ on $\mathcal{E}$, by direct computation, we can deduce that 
	\begin{equation*}\label{eq:bb7}
	u=\partial_\nu u=0\quad\mbox{on}\ \ \mathcal{E},
	\end{equation*}
	which in turn yields by the Homogren's uniqueness result (cf. \cite{Liu-Zou}) that $u=0$ in $\mathbb{R}^3\backslash\Omega$. Therefore, we arrive at the same contradiction as that in \eqref{eq:aa6}, which implies \eqref{eta}.
\end{proof}

We proceed to consider the unique determination of rational obstacles. By Definition \ref{def:irration} and Definition \ref{ir obstacle},  {\color{black}{we know that a rational obstacle contains at least one rational 3D vertex corner point}}. Recall Section \ref{sec2} and \ref{sec3}. For a fixed rational 3D vertex corner point $\mathbf{ x}_c$ which is intersected by $\Pi_i$, where $\Pi_i=\mathrm{span}\{\overrightarrow{OA_i},\overrightarrow{OA_{i+1}}\}$ with the $n$ dihedral angles $\angle(\Pi_i, \Pi_{i+1})=\alpha_i\cdot\pi$, $i=1,2,\cdots,n$, $n\geq3$, it is direct to verify that the eigenfunction $u$ to \eqref{eq:eig} of the form
\begin{equation*}\label{ra-expan}
	u(\mathbf{ x})=4\pi\sum_{n=0}^{\infty}\sum_{m=-n}^{n}\bsi^na_n^mj_n(\sqrt{\lambda}r)\sqrt{\frac{2n+1}{4\pi}\frac{(n-|m|)!}{(n+|m|)!}}P_n^{|m|}(\cos\theta)e^{\bsi m\phi}
\end{equation*}
satisfies that $a_0^0=0$ if $u(\mathbf{ x}_c)=0$,where $\mathbf{ x}=(x_1, x_2, x_3)=r(\sin\theta\cos\phi, \sin\theta\sin\phi, \cos\theta)\in\mathbb{R}^3$, $\lambda$ is the corresponding eigenvalue, $P_n^m(t)$ denotes the associated Legendre function and $j_n(t)$ signifies the $n$-th spherical Bessel function. Since $\alpha_i\in(0,1)$ for any $i=1,2,\cdots,n$, one can immediately obtain that $a_1^{\pm1}=0$; see Theorems \ref{gene-nodal}, \ref{3d vertex-thm1} and \ref{3d vertex-thm2} for detailed discussions. Moreover, if we denote 
\begin{equation}\label{edge}
\overrightarrow{OA_i}=(r,\theta_i,\phi_i)\mbox{ for }r>0, \theta_i\in(0,\pi)\mbox{ and }\phi_i\in(0,2\pi)
\end{equation}
in the spherical coordinate system, then there always holds $P_1^1(\cos\theta_i)=-\sin\theta_i\neq0$. However, since $P_1^0(\cos\theta_i)=\cos\theta_i$, we know that $P_1^0(\cos\theta_i)\neq0$ only works for $\theta_i\neq\frac{\pi}{2}$ and thus $a_1^0=0$. That means, the eigenfunction $u$ is vanishing at least to the second order when $\theta_i\neq\frac{\pi}{2}$, otherwise, $u$ is vanishing at least to the first order.

Let $\Omega$ be a polyhedron in $\mathbb{R}^3$ and $\mathbf x_c$ be a vertex-corner point of $\Omega$. Following the same notations in \cite[Theorem 8.6]{CDL2}, we define
\begin{equation*}\label{eq:region1}
\Omega_r(\mathbf x_c)=B_r(\mathbf x_c)\cap \mathbb{R}^3\backslash\overline{\Omega},\ \ r\in\mathbb{R}_+. 
\end{equation*}
For any function $f\in L_{loc}^2(\mathbb{R}^3\backslash\overline{\Omega})$, define 
\begin{equation*}\label{eq:l1}
\mathcal{L}(f)(\mathbf x_c):=\lim_{r\rightarrow+0}\frac{1}{|\Omega_r(\mathbf x_c)|}\int_{\Omega_r(\mathbf x_c)} f(\mathbf x)\ {\rm d} \mathbf x
\end{equation*}
if the limit exists. It is easy to see that if $f(\mathbf x)$ is continuous in $\overline{\Omega_{\epsilon_0}(\mathbf x_c)}$ for a sufficiently small $\epsilon_0\in\mathbb{R}_+$, then $\mathcal{L}(f)(\mathbf x_c)=f(\mathbf x_c)$.

Now we are ready to study the unique determination of rational obstacles.

\begin{theorem}\label{inverse2}
	Let $(\Omega, \eta)$ and $(\widetilde{\Omega },\widetilde{\eta})$ be two admissible complex rational obstacles of degree $p\geq 3$. Let $k\in\mathbb{R}_+$ be fixed and $\mathbf{d}_\ell$, $\ell=1, 2$ be two distinct incident directions from $\mathbb{S}^2$. Suppose $u_\ell(\mathbf  x)=u(\mathbf x; k, \mathbf{d}_\ell)$ and $\widetilde{u}_\ell=\widetilde{u}(\mathbf{ x};k,\mathbf{d}_\ell)$ are the total wave fields associated with $(\Omega, \eta)$ and $(\widetilde{\Omega},\widetilde{\eta})$ for $e^{\mathrm{i}k\mathbf x\cdot\mathbf{d}_\ell}$, respectively, and the corresponding far-field patterns are denoted by $u_{\ell,\infty}(\hat{\mathbf x};k,\mathbf{d}_\ell)$ and $\widetilde{u}_{\ell,\infty}(\hat{\mathbf x};,k,\mathbf{d}_\ell)$, $\ell=1, 2$.  Recall that $\mathbf{G}$ signifies the unbounded connected component of $\mathbb{R}^3\backslash\overline{(\Omega\cup\widetilde\Omega)}$. 
	If the following conditions are fulfilled, 
	\begin{equation}\label{cond51}
	u_{\ell,\infty}(\hat {\mathbf x}; k, \mathbf{d}_\ell )=\widetilde u_{\ell,\infty}(\hat {\mathbf x}; k, \mathbf{d}_\ell), \ \ \hat{\mathbf x}\in\mathbb{S}^2, \ell=1, 2,
	\end{equation}
	\begin{equation}\label{eq:cc1}
	\mathcal{L}\left( u_2\cdot\nabla u_1-u_1\cdot\nabla u_2\right)(\mathbf x_c)\neq 0\quad\mbox{and}\quad \mathcal{L}\left( \widetilde{u}_2\cdot\nabla \widetilde{u}_1-\widetilde{u}_1\cdot\nabla \widetilde{u}_2\right)(\mathbf x_c)\neq 0,
	\end{equation}
	where $\mathbf x_c$ is any vertex of $\Omega$, then one has that
		$$
		\left(\partial \Omega \backslash \partial \overline{ \widetilde{\Omega }} \right  )\cup \left(\partial \widetilde{\Omega } \backslash \partial \overline{ \Omega } \right )
		$$
		cannot possess a vertex corner point on  $\partial \mathbf{G}$.
\end{theorem}

\begin{proof}
	We prove the theorem by contradiction. Assume that \eqref{cond51} holds but $
	\left(\partial \Omega \backslash \partial \overline{ \widetilde{\Omega }} \right  )\cup \left(\partial \widetilde{\Omega } \backslash \partial \overline{ \Omega } \right )
	$ has a vertex-corner point $\mathbf{ x}_c$ on $\partial \mathbf{G}$. Without loss of generality, We still assume that $\mathbf{x}_c$ is a 3D vertex corner of $\widetilde{\Omega}$. In what follows, we adopt the same notation as those introduced in the proof of Theorem~\ref{inverse1}.
	
	By following a similar argument to the proof of Theorem~\ref{inverse1}, one can show that there exist $n$ pieces of planes $\Pi_i\subset\partial\mathbf{G}$ intersecting at $\mathbf{ x}_c$, such that $\partial_\nu u+\widetilde\eta u=0$ on $\Pi_i$, $i=1,2,\cdots,n$. Using the fact that $u=\widetilde u$ near $\mathbf x_c$ by Rellich Lemma and the condition \eqref{eq:cc1} on $(\widetilde\Omega,\widetilde\eta)$, we actually have
	\begin{equation}\label{eq:cc2}
	u(\mathbf x_c; k, \mathbf{d}_2)\cdot\nabla u(\mathbf x_c; k, \mathbf{d}_1)-u(\mathbf x_c; k, \mathbf{d}_1)\cdot\nabla u(\mathbf x_c; k, \mathbf{d}_2)\neq 0. 
	\end{equation}
	Clearly, \eqref{eq:cc2} implies that $\alpha_1:=u(\mathbf x_c; k, \mathbf{d}_2)$ and $\alpha_2=-u(\mathbf x_c; k, \mathbf{d}_1)$ cannot be identically zero. Let $v$ be the one introduced in \eqref{eq:bb2}, then 
	by direct verification we know that $v$ fulfills \eqref{eq:bb3} as well as
	\begin{equation}\label{eq:cc3}
	v(\mathbf x_c)=0\quad\mbox{and}\quad \nabla v(\mathbf x_c)\neq 0. 
	\end{equation}
	
	Since $\widetilde\Omega$ is rational of degree $p\geq 3$, we know that $\Pi_i$, $i=1,2,\cdots,n$, intersect either at an irrational 3D vertex corner or a rational 3D vertex corner of degree $p\geq 3$. In either case, by our results in Sections~\ref{sec2}, \ref{sec3} and \ref{sec4}, we see that $v$ is vanishing at least to second order at $\mathbf x_c$ if $\theta_i\neq\frac{\pi}{2}$ in \eqref{edge} for $i=1,2,\cdots,n$. Hence, there holds $\nabla v(\mathbf x_c)=0$, which is a contradiction to \eqref{eq:cc3}. 
	%
\end{proof}

\begin{remark}
	In the proof of this theorem, we illustrate the vanishing order of $u$ by the normal derivatives in Taylor expansion. That is, $v(\mathbf{ x}_c)=\nabla v(\mathbf{ x}_c)=0$ implies that $v$ is vanishing at $\mathbf{ x}_c$ at least to the second order. Indeed, it is equivalent to know that $a_0^0=a_1^{\pm1}=a_1^0=0$ of $u$ in spherical wave expansion, which can be easily verified in the main theorems of Sections \ref{sec2} and \ref{sec3} under the condition that $\theta_i\neq\frac{\pi}{2}$, $i=1,2,\cdots,n$.
\end{remark}

 

%

\begin{remark}
	The uniqueness results and the corresponding argument in Theorems \ref{inverse1} and \ref{inverse2} are ``localized" around the corner $\mathbf{ x}_c$ based on the spectral results in Sections \ref{sec2} and \ref{sec3}, which is more effective in the application of inverse problems.
\end{remark}

We would like to point out that the condition \eqref{eq:cc1} can actually be fulfilled if one imposes certain generic conditions on $\Omega$. For instance, if the obstacle $\Omega$ is sufficiently small compared to the wavelength, namely $k\cdot\mathrm{diam}(\Omega)\ll 1$, then from a physical point of view, the scattered wave field due to the obstacle is of a much smaller magnitude than the incident field, and thus the incident plane wave dominates in the total wave field $u=u^i+u^s$. Under this circumstance, \eqref{eq:cc1} can be verified  straightforwardly. 

%

\section*{Acknowledgement}

The work  of H Diao was supported in part by the Fundamental Research Funds for the Central Universities under the grant 2412017FZ007.  The work of H Liu was supported by the FRG fund from Hong Kong Baptist University and the Hong Kong RGC General Research Fund (projects 12301218 and 12302017).  The work of J Zou 
was supported by the Hong Kong RGC General Research Fund (project 14304517) and 
NSFC/Hong Kong RGC Joint Research Scheme 2016/17 (project N\_CUHK437/16).


\begin{thebibliography}{99}

%


%
  
%
%
\bibitem{AR} G.~Alessandrini and L.~Rondi, {\it Determining a sound-soft polyhedral scatterer by a single far-field measurement}, Proc. Amer. Math. Soc., {\bf 35} (2005), 1685--1691.
%
%
%





%
%
%




%
%
%
%
%
%
%
%
%

\bibitem{WB}
W. Bosch, {\it On the computation of derivatives of Legendre functions}, Phys. Chem. Earth, {\bf 25}(9--11), 655--659.





%

%





  
  
%
%
%

\bibitem{CDL2}
X.~Cao, H.~Diao, H.~Liu and J.~Zou, {\it On nodal and generalized singular structures of Laplacian eigenfunctions and applications to inverse scattering problems}, arXiv:1902.05798, 2019. 



\bibitem{CY}
J.~Cheng and M.~Yamamoto, {\it Uniqueness in an inverse scattering problem within non-trapping polygonal obstacles with at most two incoming waves},
Inverse Problems, {\bf 19} (2003), 1361--1384.



\bibitem{CK}
D.~Colton and R.~Kress, {\it Inverse Acoustic and Electromagnetic Scattering Theory}, 3rd edition, Springer-Verlag, Berlin, 2013.

\bibitem{CK18} D. Colton and R. Kress, {\it Looking back on inverse scattering theory}, SIAM Review, {\bf 60} (2018), no. 40, 779--807. 

%
%
%




%
%
%
%
%
%
%
%
%
%





%
%
%
%



%
%
%





%
%
%



%
%








%
%
%
%









%









%
%





%
%


%

  
  
\bibitem{LP} P. Lax and R. Phillips, {\it Scattering Theory}, Academic Press, New York and London, 1967. 
%










%
%
%
%
%
%
%
%


\bibitem{LPRX} H. Liu, M. Petrini, L. Rondi, Luca and J. Xiao, {Stable determination of sound-hard polyhedral scatterers by a minimal number of scattering measurements}, {\em J. Differential Equations}, {\bf 262} (2017), no. 3, 1631--1670.

\bibitem{LRX} {H. Liu, L. Rondi and J. Xiao}, {Mosco convergence for $H(curl)$ spaces, higher integrability for Maxwell's equations, and stability in direct and inverse EM scattering problems}, {\em J. Eur. Math. Soc. (JEMS)}, \textbf{21} (2019), 2945--2993.  

\bibitem{Liu-Zou}
H.~Liu and J.~Zou,
{\it Uniqueness in an inverse acoustic obstacle scattering problem for both sound-hard and sound-soft polyhedral scatterers},
Inverse Problems, {\bf 22} (2006),
515--524.

\bibitem{Liu-Zou3}
H.~Liu and J.~Zou,
{\it On unique determination of partially coated polyhedral scatterers
with far field measurements},
Inverse Problems, {\bf 23} (2007), 297--308.

%
%
%

\bibitem{Mclean}
W. Mclean, {\em Strongly Elliptic Systems and Boundary Integral Equation},
Cambridge University Press, Cambridge, 2000.


%
%
%
%
%


%
%
%






\bibitem{Ned} J.-C.~N\'ed\'elec, {\it Acoustic and Electromagnetic Equations}, Springer-Verlag, New York, 2001.
%
%
%
%
%
%
%
%
%




%
%
%
%
%
%









%
%
%
%
%
%


%
%
%
%
%
%
%
%





%
%
%
%



%
%
%


\end{thebibliography}
\end{document}